%% file: euler_av_num_final.tex
\documentclass{dcds-b}
\usepackage{amsmath}
  \usepackage{paralist}
  \usepackage{graphics} 
  \usepackage{epsfig} 
\usepackage{graphicx}  \usepackage{epstopdf}
\usepackage{color}
\usepackage{cases}
 \usepackage[colorlinks=true]{hyperref}
\hypersetup{urlcolor=blue, citecolor=red}

\def\currentvolume{20}
 \def\currentissue{4}
  \def\currentyear{2015}
   \def\currentmonth{June}
    \def\ppages{X--XX}
     \def\DOI{10.3934/dcdsb.2015.20.xx}

\newtheorem{corollary}{Corollary}

\newtheorem{proposition}{Proposition}

\theoremstyle{definition}

\newtheorem{remark}{Remark}

\newcommand{\avg}[1]{\langle{#1}\rangle}
\definecolor{Red}{rgb}{1,0,0}
\definecolor{Blue}{rgb}{0,0,1}
\newcommand{\red}[1]{{#1}} 
\newcommand\R{\mathbb{R}}
\newcommand\1{\bf 1}
\newcommand\sech{\,\mbox{sech}}

\graphicspath{{Figures/}}

\title[An energy-consistent depth-averaged Euler system] 
      {An energy-consistent depth-averaged Euler system: Derivation
  and properties}

\author[M.-O.~Bristeau, A.~Mangeney, J.~Sainte-Marie and N.~Seguin]{}


 \email{Marie-Odile.Bristeau@inria.fr}
 \email{mangeney@ipgp.fr}
\email{Jacques.Sainte-Marie@inria.fr}
 \email{Nicolas.Seguin@upmc.fr}

\begin{document}
\maketitle

\centerline{\scshape Marie-Odile Bristeau$^{1, 2, 3, 4}$,  Anne
  Mangeney$^{5, 1, 2, 3, 4}$}
\centerline{\scshape Jacques Sainte-Marie$^{2, 1, 3, 4}$ and Nicolas Seguin$^{3, 1, 2, 4}$}
\medskip
{\footnotesize
 \centerline{$^{1}$Inria, EPC ANGE, Rocquencourt - B.P. 105}
  \centerline{F78153 Le Chesnay cedex, France}
  \medskip
 \centerline{$^{2}$CEREMA,  134 rue de Beauvais}
  \centerline{F-60280 Margny-L\`es-Compi\`egne, France}
    \medskip
  \centerline{$^{3}$Sorbonne Universit\'es, UPMC Univ Paris 06, UMR 7598}
 \centerline{Laboratoire Jacques-Louis Lions, F-75005, Paris, France}
   \medskip
\centerline{$^{4}$CNRS, UMR 7598, Laboratoire Jacques-Louis Lions}
\centerline{F-75005, Paris, France}
  \medskip
\centerline{$^{5}$University Paris Diderot, Sorbonne Paris Cit\'e}
\centerline{Institut de Physique du Globe de Paris}
\centerline{Seismology group, 1 rue Jussieu, 75005 Paris, France}
} 

%
%

\bigskip

 \centerline{(Communicated by Benoit Perthame)}

\begin{abstract}
\red{In this paper, we present an original derivation process of a
  non-hydrostatic shallow water-type model which aims at approximating
  the incompressible Euler and Navier-Stokes systems with free
  surface.} The closure relations are obtained by a
minimal energy constraint instead of an asymptotic expansion. The model slightly differs from the
well-known Green-Naghdi model and is confronted with stationary and
analytical solutions of the Euler system corresponding to rotational
flows. \red{At the end of the paper}, we give
time-dependent analytical solutions for the Euler system that are also
analytical solutions for the proposed model but that are not solutions
of the  Green-Naghdi model. \red{We also give and compare analytical solutions of the
two non-hydrostatic shallow water models.}
\end{abstract}

\section{Introduction}

Despite the progress in the analysis and numerical approximation of
the incompressible Euler and Navier-Stokes equations with free
surface, there exists a demand for models of reduced
complexity such as shallow water type models to represent gravity
driven geophysical flows. In particular, the accurate description
of the topography or bathymetry that play a key role in landslide
dynamics or ocean wave propagation, requires simplified models
to reduce the associated high computational cost.

Non-linear shallow water equations model the dynamics of a shallow,
rotating layer of homogeneous incompressible fluid and are typically
used to describe vertically averaged flows in two or three dimensional domains
in terms of horizontal velocity and depth variations. The classical Saint-Venant system \cite{saint-venant}
with viscosity and friction \cite{saleri,gerbeau,marche} is particularly well-suited for the
study and numerical simulations of a large class of geophysical phenomena
such as rivers, \red{lava flows, ice sheets,} coastal domains, oceans or even run-off or avalanches
when being modified with adapted source terms \cite{Bouchut2003531,bouchut,mangeney07}. But the Saint-Venant system is built on the hydrostatic
assumption  consisting in neglecting the
vertical acceleration of the fluid. This assumption is valid for a large class of geophysical flows but is
restrictive in various situations where the dispersive effects~-- such as
those occuring in wave propagation~-- cannot
be neglected. As an example, neglecting the vertical acceleration in granular
flows or landslides leads to significantly overestimate the initial flow velocity
\cite{mangeney05,mangeney_nature1}, with strong implication for hazard assessment.

The modeling of the non-hydrostatic effects for shallow water flows
does not raise insuperable
difficulties~\cite{green,camassa1,bbm,nwogu,peregrine,JSM_DCDS} but the analysis
\cite{lannes,li} of the resulting models and their discretization become
tough. The assumption of potential flows is often used to derive
dispersive models and an extensive literature exists concerning
these models. The most important contributions have been
proposed by Lannes and
co-authors~\cite{chazel1,chazel2,bonneton_lannes,lannes,lannes1}, see also~\cite{dutykh}.

The non-hydrostatic model presented in this paper is not based on the
irrotational assumption, on the other hand it is not derived using
an asymptotic expansion of the incompressible Navier-Stokes or Euler
based on the classical shallow water assumptions. Even if such an
asymptotic expansion approach is natural, it leads to difficulties for the approximation of
the non-hydrostatic pressure terms.

To overcome these problems, we propose a strategy for the model
derivation that is widely used in the kinetic framework to obtain
kinetic descriptions e.g. of conservations
laws~\cite{levermore,perthame}. The required closure relations to
obtain a depth-averaged model approximating the Euler or Navier-Stokes
system satisfy an energy-based optimality criterion. As a
consequence, the proposed model slightly differs from existing models
especially the well-known Green-Naghdi model~\cite{green,li}. It consists
in a set of first order partial differential equations and compared to
the Green-Naghdi model, the contribution of the non-hydrostatic
pressure terms differs from a scaling coefficient. Illustrating these differences, we give
time-dependent analytical solutions for the Euler system that are also
analytical solutions for the proposed model but that are not solution
of the  Green-Naghdi model.

The discretization of the proposed model is not in the scope of this
paper, we only notice that numerical techniques have been recently
proposed for the approximation of non-hydrostatic models but their properties (numerical cost/robustness) are not
fully satisfactory~\cite{chazel2,JSM_CF,lemetayer} for practical uses especially
in 2d with unstructured meshes. Since this model has the structure of a
conservation law with additional terms and only contains first order
derivatives, we hope that it can be
discretized more easily using finite volume techniques.

The paper is organized as follows. In Section~\ref{sec:NS}, we recall
the incompressible Navier-Stokes equations with free surface with the
associated boundary conditions and we deduce the Euler system. In Section~\ref{sec:av_euler} we
derive the proposed non-hydrostatic model. Some of its properties
are investigated in Section~\ref{sec:properties} and confrontations with analytical
solutions are given in Section~\ref{sec:anal_sol}.

\section{The Navier-Stokes and Euler systems}
\label{sec:NS}

In this section, we present the Navier-Stokes and Euler systems with
their associated boundary conditions.

\subsection{The Navier-Stokes equations}


The Navier-Stokes equations restricted to two dimensions have the
following general formulation

\ \vspace*{-10pt}
\begin{subnumcases}{\label{eq:navier-stokes}}
\frac{\partial u}{\partial x} + \frac{\partial w}{\partial z} = 0,\label{eq:NS_2d1}\\
\frac{\partial u}{\partial t} + u \frac{\partial u}{\partial x} + w \frac{\partial u}{\partial z} + \frac{\partial p}{\partial x} = \frac{\partial \Sigma_{xx}}{\partial x} + \frac{\partial \Sigma_{xz}}{\partial z},\label{eq:NS_2d2}\\
\frac{\partial w}{\partial t} + u\frac{\partial w}{\partial x} + w\frac{\partial w}{\partial z} + \frac{\partial p}{\partial z} = -g + \frac{\partial \Sigma_{zx}}{\partial x} + \frac{\partial \Sigma_{zz}}{\partial z},
\label{eq:NS_2d3}
\end{subnumcases}
where the $z$ axis represents the vertical
direction. We consider this system for
$$t > t_0, \quad x \in \R, \quad z_b(x,t) \leq z \leq \eta(x,t),$$
where $\eta(x,t)$ represents the free surface elevation, ${\bf u}=(u,w)^T$ the horizontal and vertical velocities. The water
height is $H = \eta - z_b$, see Fig.~\ref{fig:notations}. We consider that the
bathymetry $z_b$ can vary with respect to abscissa $x$ and also with
respect to time $t$. The chosen form of the viscosity stress tensor is symmetric
\begin{eqnarray*}
 \Sigma_{xx} = 2 \mu \frac{\partial u}{\partial x}, & \qquad & \Sigma_{xz} = \mu \bigl( \frac{\partial u}{\partial z} + \frac{\partial w}{\partial x} \bigr),\label{eq:visco1}\\
\Sigma_{zz} = 2 \mu \frac{\partial w}{\partial z}, & \qquad & \Sigma_{zx} = \mu \bigl( \frac{\partial u}{\partial z} + \frac{\partial w}{\partial x}\bigr),
\label{eq:visco2}
\end{eqnarray*}
with $\mu$ the viscosity that is supposed constant. For a more general form of the viscosity tensor, see Ref.~\cite{decoene,levermore}. We define the
total stress tensor $\Sigma_T$
$$\Sigma_T = -p I_d + \Sigma.$$

\begin{figure}[pb]
\begin{center}
\resizebox{9cm}{!}{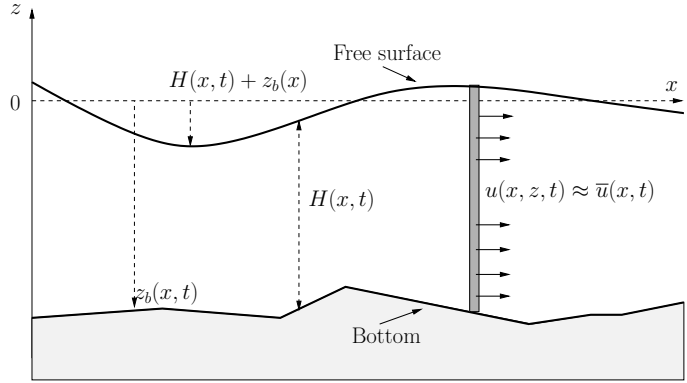}
\caption{Notations: water height $H(x,t)$, free surface $\eta(x,t)$ and bottom $z_b(x,t)$.}
\end{center}
\label{fig:notations}
\end{figure}
As in Ref.~\cite{gerbeau}, we introduce the indicator function for the fluid
region
\begin{equation}
\varphi(x,z,t) = \left\{\begin{array}{ll}
1 & \hbox{ for } (x,z) \in \Omega = \{(x,z)\,|\, z_b \leq z \leq \eta\},\\
0 & \hbox{ otherwise.}
\end{array}\right.
\label{eq:varphi}
\end{equation}
The fluid region is advected by the flow, which can be expressed, thanks
to the incompressibility condition, by the relation
\begin{equation}
\frac{\partial \varphi}{\partial t} + \frac{\partial \varphi u}{\partial
x} + \frac{\partial \varphi w}{\partial z} = 0.
\label{eq:advec_phi}
\end{equation}
The solution $\varphi$ of this equation takes the values 0 and 1 only
but it needs not be of the form (\ref{eq:varphi}) at all times. The
analysis below is limited to the conditions where this form is
preserved. For a more complete presentation of the Navier-Stokes
system and its closure, the reader can refer to~\cite{lions}.

\begin{remark}
Notice that in the fluid domain,
Eq.~(\ref{eq:advec_phi}) reduces to the divergence free condition whereas
across the upper and lower boundaries it gives the
kinematic boundary conditions defined in the following.
\end{remark}


\subsection{Boundary conditions}
\label{subsec:BC}

The system \eqref{eq:navier-stokes} is completed with boundary conditions. The outward and upward unit normals to the free surface ${\bf n}_s$ and to the bottom ${\bf n}_b$ are
given by
$${\bf n}_s = \frac{1}{\sqrt{1 + \bigl(\frac{\partial \eta}{\partial x}\bigr)^2}}  \left(\begin{array}{c} -\frac{\partial \eta}{\partial x}\\ 1 \end{array} \right), \quad {\bf n}_b = \frac{1}{\sqrt{1 + \bigl(\frac{\partial z_b}{\partial x}\bigr)^2}}  \left(\begin{array}{c} -\frac{\partial z_b}{\partial x}\\ 1 \end{array} \right).$$

\subsubsection{At the free surface}

Classically at the free surface we have the kinematic boundary condition
\begin{equation}
\frac{\partial \eta}{\partial t} + u_s \frac{\partial \eta}{\partial x}
-w_s = 0,
\label{eq:free_surf}
\end{equation}
where the subscript $s$ denotes the value of the
considered quantity at the free surface. The dynamical condition at the free surface takes into account the
equilibrium with the atmospheric pressure. Considering the air viscosity is
negligible, the continuity
of stresses at the free boundary imposes
\begin{equation}
\Sigma_T {\bf n}_s = -p^a(x,\eta(x,t),t) {\bf n}_s,
\label{eq:BC2}
\end{equation}
where $p^a=p^a(x,z,t)$ is a given function corresponding to the
atmospheric pressure.

\subsubsection{At the bottom}

Since we consider that the bottom can vary with respect to time $t$, the
kinematic boundary condition is
\begin{equation}
\frac{\partial z_b}{\partial t} + u_b \frac{\partial z_b}{\partial x}
-w_b = 0,
\label{eq:bottom}
\end{equation}
where the subscript $b$ denotes the value of the
considered quantity at the bottom and $(x,t) \mapsto z_b(x,t)$ is a
given function. Notice that Eq.~(\ref{eq:bottom})
reduces to a classical no-penetration condition when $z_b$
does not depend on time $t$.

For the stresses at the bottom we consider a friction law under the form
\begin{equation}
\Sigma_T {\bf n}_b - ({\bf n}_b . \Sigma_T {\bf n}_b){\bf n}_b = \kappa {\bf v}_b,
\label{eq:BC_z_b}
\end{equation}
with ${\bf v}_b={\bf u}_b - (0,\frac{\partial z_b}{\partial t})^T$ the
relative velocity between the water and the bottom and $\kappa$ is a positive friction coefficient. Let
  ${\bf t}_b$ satisfies $({\bf t}_b)^t {\bf n}_b = 0$ then after
  multiplication by ${\bf n}_b$, Eq.~(\ref{eq:BC_z_b}) leads to
$$({\bf v}_b)^t {\bf n}_b = 0,$$
that is equivalent to Eq.~(\ref{eq:free_surf}). Similarly multiplying
Eq.~(\ref{eq:BC_z_b}) by ${\bf t}_b$ gives
\begin{equation}
({\bf t}_b)^t \Sigma_T {\bf n}_b = \kappa ({\bf v}_b)^t {\bf
  t}_b = \kappa \left( 1+ \left(\frac{\partial z_b}{\partial x}
  \right)^2\right)u_b.
\label{eq:navier_simple}
\end{equation}

\subsection{Energy balance}

We recall the fondamental stability property related to the fact that
the Navier-Stokes system admits an energy
\begin{equation}
E = E(z;u,w) = \frac{u^2+w^2}{2} + gz,
\label{eq:energy_exp}
\end{equation}
leading to the following equation
\begin{equation}
\begin{split}
\frac{\partial}{\partial t}\int_{z_b}^{\eta} (E + p^a)\ dz +&
\frac{\partial}{\partial x}\int_{z_b}^{\eta} \left[ u\bigl( E + p
  \bigr) -\mu \left( 2u\frac{\partial u}{\partial x} + w\left(
      \frac{\partial u}{\partial z} + \frac{\partial w}{\partial x}\right)\right)\right] dz
=\\
& -2\mu\int \left[ \left(\frac{\partial u}{\partial x}\right)^2 +
  \frac{1}{2}\left(\frac{\partial u}{\partial z} + \frac{\partial
      w}{\partial x}\right)^2 + \left(\frac{\partial w}{\partial
      z}\right)^2 \right] dz\\
 &
+ H\frac{\partial p^a}{\partial t}
+
(\left. p\right|_b-p^a)\frac{\partial z_b}{\partial t} - \kappa u_b.
\label{eq:energy_eq}
\end{split}
\end{equation}

\subsection{The Euler system}

Neglecting the viscous effects, we consider the Euler equations written in a conservative form
\begin{subnumcases}{\label{eq:euler}}
\frac{\partial {\varphi}}{\partial {t}} + \frac{\partial {\varphi
    u}}{\partial {x}} + \frac{\partial {\varphi w}}{\partial {z}} = 0,\label{eq:eul_2d1}\\
\frac{\partial {\varphi  u}}{\partial {t}} + \frac{\partial \varphi{u}^2}{\partial {x}} + \frac{\partial \varphi{u}{w}}{\partial {z}} + \varphi\frac{\partial {p}}{\partial {x}} = 0,\label{eq:eul_2d2}\\
\frac{\partial {\varphi  w}}{\partial {t}} + \frac{\partial
  \varphi {u}{w}}{\partial {x}} + \frac{\partial \varphi
  {w}^2}{\partial {z}}  + \varphi \frac{\partial {p}}{\partial {z}} =
-\varphi  g,\label{eq:eul_2d3}
\end{subnumcases}
with $\varphi$ defined by~\eqref{eq:varphi}. The energy equation
writes
\begin{equation}
\frac{\partial}{\partial t}\int_{z_b}^{\eta} (E + p^a)\ dz +
\frac{\partial}{\partial x}\int_{z_b}^{\eta}  u\bigl( E + p
  \bigr)
=  H\frac{\partial p^a}{\partial t} +
(\left. p\right|_b-p^a)\frac{\partial z_b}{\partial t},
\label{eq:energy_eq_euler}
\end{equation}
with $E$ defined by~\eqref{eq:energy_exp}. This system is completed with the boundary conditions~\eqref{eq:free_surf},\eqref{eq:BC2}
and~\eqref{eq:bottom}. In our case, \eqref{eq:BC2} reduces to
\begin{equation}
\left. p \right|_s = p^a.
\label{eq:ps}
\end{equation}

For the sake of simplicity, in the following we neglect the
variations of the atmospheric pressure $p^a$ i.e. $p^a =
p^a_0$ with $p^a_0=0$.

\subsection{Non negativity of the pressure}
\label{subsec:assump}

We also suppose in each point of the fluid region~-- including at the bottom~-- we have
$$p - p^a \geq 0 .$$
The analysis below and especially the kinetic interpretation is restricted to this situation.
Notice that in the case of hydrostatic Euler equations since we have
$$p - p^a = g(\eta -z),$$
this assumption reduces to the non-negativity of the water height $H$.

\section{Depth-averaged solutions of the Euler and Navier-Stokes systems}
\label{sec:av_euler}

In this section we take the vertical average of the Euler system and
study  the necessary closure relations for this system.

Let us denote $\avg{f}$ the average along the vertical axis, the
so-called \emph{depth-average}, of the quantity $f=f(z)$
i.e.
\begin{equation}
\avg{f}(x,t) = \int_\R f(x,z,t) \ dz.
\label{eq:mean1}
\end{equation}
\red{During the derivation process of the model, we assume the bottom topography does not depend on time $t$, i.e.
$$\frac{\partial z_b}{\partial t} = 0.$$
The contribution of the time variations of the bottom topography is
given in remark~\ref{rem:zb_t}.}

\subsection{Depth-averaging of the Euler solution}
\label{subsec:depth-averaging}

The goal is to transpose the
entropy-based moment closures proposed by Levermore in
\cite{levermore_closure} for kinetic equations to our framework. In such a
way, we obtain a nonperturbative derivation of shallow-water models
which is justified by an entropy minimization process under constraint.
The constraints concern the moments of the solution of the Euler
equation, which are here the depth-averaged variables.

Taking into account the kinematic boundary conditions~\eqref{eq:free_surf}
and~\eqref{eq:bottom}, the depth-averaged form of the Euler system~\eqref{eq:euler}
writes
\begin{subnumcases}{\label{eq:euler_av}}
  \frac{\partial}{\partial t} \avg{\varphi} + \frac{\partial}{\partial x} \avg{\varphi u} = 0 , \label{eq:av1}\\
  \frac{\partial}{\partial t} \avg{\varphi u} +
  \frac{\partial}{\partial x} \avg{\varphi u^2} + \avg{\varphi
    \frac{\partial p}{\partial x} } = 0 , \label{eq:av2}\\
  \frac{\partial}{\partial t} \avg{\varphi w} +
  \frac{\partial}{\partial x} \avg{\varphi u w} + \avg{\varphi
    \frac{\partial p}{\partial z}} = -\avg{\varphi g},\label{eq:av3}\\
\frac{\partial}{\partial t} \avg{\varphi z} + \frac{\partial}{\partial
  x} \avg{\varphi z u} = \avg{\varphi w},\label{eq:av4}
\end{subnumcases}
where the last equation is a rewriting of
$$\avg{ \int_{z_b}^z \left(\frac{\partial {\varphi}}{\partial {t}} + \frac{\partial {\varphi
    u}}{\partial {x}} + \frac{\partial {\varphi w}}{\partial
  {z}}\right) dz} = \avg{z \left( \frac{\partial {\varphi}}{\partial {t}} + \frac{\partial {\varphi
    u}}{\partial {x}} + \frac{\partial {\varphi w}}{\partial {z}}\right)}
 =0,$$
using again the kinematic boundary conditions. Notice that using the
definition~\eqref{eq:varphi}, we have
\begin{equation}
\avg{\varphi} = H,\quad\text{and }\; \avg{\varphi z} = \frac{\eta^2 -
  z_b^2}{2}.
\label{eq:moments01}
\end{equation}
Simple manipulations allow to obtain the
system~\eqref{eq:euler_av} from the Euler
system~\eqref{eq:euler},\eqref{eq:free_surf} and~\eqref{eq:bottom} e.g. for Eq.~\eqref{eq:av1},
starting from~\eqref{eq:eul_2d1} we write
$$\avg{\frac{\partial \varphi}{\partial t} + \frac{\partial \varphi u}{\partial
    x} + \frac{\partial \varphi w}{\partial z}} = 0,$$
and permuting the derivative with the integral using the Leibniz rule
directly gives~\eqref{eq:av1}.

We decompose the pressure $p$ under the form
$$p = g(\eta-z) + p_{nh},$$
i.e. the sum of the hydrostatic and non-hydrostatic parts of the
pressure. Hence, the system~\eqref{eq:euler_av} becomes
\begin{subnumcases}{\label{eq:euler_av1}}
  \frac{\partial}{\partial t} \avg{\varphi } + \frac{\partial}{\partial x} \avg{\varphi u} = 0 , \label{eq:av11}\\
  \frac{\partial}{\partial t} \avg{\varphi u} +
  \frac{\partial}{\partial x} \left( \avg{\varphi u^2} +
    g\avg{\varphi (\eta-z)} + \avg{\varphi p_{nh}}\right)  =  -\left(g \avg{\varphi} + \left. p_{nh}\right|_b\right)
  \frac{\partial z_b}{\partial x} ,\label{eq:av22} \\
  \frac{\partial}{\partial t} \avg{\varphi w} +
  \frac{\partial}{\partial x} \avg{\varphi u w} =
  \left. p_{nh}\right|_b,\label{eq:av33}\\
\frac{\partial}{\partial t} \avg{\varphi z} + \frac{\partial}{\partial
  x} \avg{\varphi z u} = \avg{\varphi w},\label{eq:av44}
\end{subnumcases}
where the boundary condition~\eqref{eq:ps} has been used. The energy
equation~\eqref{eq:energy_eq_euler} gives
\begin{equation}
  \frac{\partial}{\partial t} \avg{\varphi E} +
  \frac{\partial}{\partial x} \avg{\varphi u(E+p)} = 0,
\label{eq:energy_av}
\end{equation}
where $E(z;u,w)$ is defined by~\eqref{eq:energy_exp}.

Therefore the
system~\eqref{eq:euler_av1} has four equations with four
unknowns, namely $\avg{\varphi}$, $\avg{\varphi u}$, $\avg{\varphi w}$
and $\avg{\varphi p_{nh}}$ and closure relations are needed to define $\avg{\varphi u^2}$, $\avg{\varphi u w}$,
$\avg{\varphi z u}$ and $p_{nh}|_b$.

If $u'$,$w'$ are defined as the
deviations of $u$,$w$ with respect to their depth-averages, then it comes
\begin{equation}
  \label{eq:uaverage}
  \varphi u = \varphi\frac{\avg{\varphi u}}{\avg{\varphi}} + \varphi u',\quad \varphi w =
 \varphi \frac{\avg{\varphi w}}{\avg{\varphi}} + \varphi w',
\end{equation}
with $\avg{\varphi u'} = \avg{\varphi w'} = 0$. Following the moment closure proposed by Levermore~\cite{levermore_closure}, we
study the minimization problem
\begin{equation}
  \label{eq:minSW}
\min_{u',w'} \avg{ \{ \varphi E(z;u,w)  \}}.
\end{equation}
The energy $E(z;u,w) $ being quadratic with respect
to $u$ we notice
\begin{equation}
  \label{eq:uaverage2}
  \begin{aligned}
    \avg{\varphi u^2} & = \frac{\avg{\varphi u}^2}{\avg{\varphi}} + 2\avg{\varphi uu'} +
    \avg{\varphi (u')^2} \\
    &= \frac{\avg{\varphi u}^2}{\avg{\varphi}} + \avg{\varphi(u')^2} \\
    &\geq \frac{\avg{\varphi u}^2}{\avg{\varphi}},
  \end{aligned}
\end{equation}
and similarly, we obtain
\begin{equation}
  \label{eq:waverage2}
\avg{\varphi w^2} \geq \frac{\avg{\varphi w}^2}{\avg{\varphi}}.
\end{equation}
Eqs.~\eqref{eq:uaverage2} and~\eqref{eq:waverage2} mean that the solution
of the minimization problem~\eqref{eq:minSW} is given by
\begin{equation}
  \label{eq:minSW1}
\avg{\varphi E\left(z;\frac{\avg{\varphi
        u}}{\avg{\varphi}},\frac{\avg{\varphi
        w}}{\avg{\varphi}}\right)} = \min_{u',w'} \avg{ \{ \varphi E(z;u,w)  \}},
\end{equation}
and
\begin{equation}
  \label{eq:tildeE}
  \avg{\varphi E\left(z;\frac{\avg{\varphi u}}{\avg{\varphi}},\frac{\avg{\varphi w}}{\avg{\varphi}}\right)} =
  \frac{\avg{\varphi u}^2 + \avg{\varphi w}^2}{2\avg{\varphi}} + g \avg{\varphi z},
\end{equation}
Since the only choice leading to equalities in
relations~\eqref{eq:uaverage2} and~\eqref{eq:waverage2} corresponds to
\begin{equation}
u = \frac{\avg{\varphi u}}{\avg{\varphi}},\quad\text{and }\; w =
\frac{\avg{\varphi w}}{\avg{\varphi}},
\label{eq:closures}
\end{equation}
this allows to precise the closure relations associated to a minimal
energy, namely
\begin{subnumcases}{\label{eq:closure}}
\avg{\varphi u^2} = \frac{\avg{\varphi u}^2}{\avg{\varphi}},\label{eq:close1}\\
\avg{\varphi u w}  = \frac{\avg{\varphi u}\avg{\varphi w}}{\avg{\varphi}},\label{eq:close2}\\
\avg{\varphi z u} = \avg{\varphi z }\frac{\avg{\varphi u}}{\avg{\varphi }}. \label{eq:close3}
\end{subnumcases}
Replacing~\eqref{eq:closure} into Eqs.~\eqref{eq:euler_av1} leads to the system
\begin{subnumcases}{\label{eq:euler_av_phi}}
  \frac{\partial}{\partial t} \avg{\varphi } + \frac{\partial}{\partial x} \avg{\varphi u} = 0 ,\label{eq:phi1}\\
  \frac{\partial}{\partial t} \avg{\varphi u} +
  \frac{\partial}{\partial x} \left( \frac{\avg{\varphi u}^2}{\avg{\varphi}} +
    g\avg{\varphi (\eta-z)} + \avg{\varphi p_{nh}}\right)  =  \nonumber\\
\hspace*{7cm} -\left(g \avg{\varphi} + \left. p_{nh}\right|_b\right)
  \frac{\partial z_b}{\partial x} ,\label{eq:phi2}\\
 \frac{\partial}{\partial t} \avg{\varphi w} +
  \frac{\partial}{\partial x} \avg{\varphi w}\frac{\avg{\varphi u}}{\avg{\varphi}} =
  \left. p_{nh}\right|_b,\label{eq:phi3}\\
\frac{\partial}{\partial t} \avg{\varphi z} + \frac{\partial}{\partial
  x} \avg{\varphi z}\frac{\avg{\varphi u}}{\avg{\varphi}} = \avg{\varphi w},\label{eq:phi4}
\end{subnumcases}
but it remains to find the closure relation for the
non-hydrostatic pressure terms. As proved in the following
proposition, the only possible choice is
\begin{equation}
\left. p_{nh}\right|_b = 2 \frac{\avg{\varphi p_{nh}}}{\avg{\varphi}}.
\label{eq:choice_pb}
\end{equation}

\begin{proposition}
The solutions of the
Euler system~\eqref{eq:euler}-\eqref{eq:ps},\eqref{eq:bottom},\eqref{eq:free_surf} satisfying the
closure relations~\eqref{eq:closure},\eqref{eq:choice_pb} are also
solutions of the system
\begin{subnumcases}{\label{eq:euler_av11}}
  \frac{\partial}{\partial t} \avg{\varphi } + \frac{\partial}{\partial x} \avg{\varphi u} = 0 , \label{eq:av111}\\
  \frac{\partial}{\partial t} \avg{\varphi u} +
  \frac{\partial}{\partial x} \left( \frac{\avg{\varphi u}^2}{\avg{\varphi}} +
    g\avg{\varphi (\eta-z)} + \avg{\varphi p_{nh}}\right)  =
  \nonumber\\
\hspace*{6.5cm} -\left(g \avg{\varphi} + 2 \frac{\avg{\varphi p_{nh}}}{\avg{\varphi}}\right)
  \frac{\partial z_b}{\partial x} ,\label{eq:av222} \\
  \frac{\partial}{\partial t} \avg{\varphi w} +
  \frac{\partial}{\partial x} \avg{\varphi w}\frac{\avg{\varphi u}}{\avg{\varphi}} =
  2 \frac{\avg{\varphi p_{nh}}}{\avg{\varphi}},\label{eq:av333}\\
\frac{\partial}{\partial t} \avg{\varphi z} + \frac{\partial}{\partial
  x} \avg{\varphi z}\frac{\avg{\varphi u}}{\avg{\varphi}} = \avg{\varphi w}.\label{eq:av444}
\end{subnumcases}
This system is a depth-averaged approximation of the
Euler system and
admits~-- for smooth solutions~-- an energy balance
under the form
\begin{equation}
\begin{split}
  \frac{\partial}{\partial t} \avg{\varphi E
    \left(z;\frac{\avg{\varphi u}}{\avg{\varphi}},\frac{\avg{\varphi
          w}}{\avg{\varphi}}\right)}& \\
 & \hspace*{-1cm} +\frac{\partial}{\partial x} \avg{\frac{\avg{\varphi
        u}}{\avg{\varphi}} \left( \varphi E \left(z;\frac{\avg{\varphi
      u}}{\avg{\varphi}},\frac{\avg{\varphi w}}{\avg{\varphi}}\right)
+\avg{\varphi p_{nh}} \right)} = 0.
\end{split}
\label{eq:energy_av_fin}
\end{equation}
\label{prop:prop_ad}
\end{proposition}
\begin{remark}
It is important to notice that whereas the solution $H,u,w,p$ of the
Euler system~\eqref{eq:euler}-\eqref{eq:ps},\eqref{eq:bottom},\eqref{eq:free_surf} also satisfies the
system~\eqref{eq:euler_av1},
only the solutions $H,u,w,p$ of the
Euler system~\eqref{eq:euler}-\eqref{eq:ps},\eqref{eq:bottom},\eqref{eq:free_surf} satisfying the
closure relations~\eqref{eq:closure},\eqref{eq:choice_pb} are also
solutions of the system~\eqref{eq:euler_av11}-\eqref{eq:energy_av_fin}. On the
contrary, any solutions $\avg{\varphi}$, $\avg{\varphi u}$,
$\avg{\varphi w}$ and $\avg{p_{nh}}$
of~\eqref{eq:euler_av11}-\eqref{eq:av444} with $\left. p_{nh}\right|_b$
defined by~\eqref{eq:choice_pb} are also
solutions of~\eqref{eq:euler_av1}-\eqref{eq:energy_av}.
\end{remark}

\begin{proof}[Proof of prop.~\ref{prop:prop_ad}]
Only the manipulations allowing to obtain~\eqref{eq:energy_av_fin}
have to be detailed. More precisely, we have to prove that, in~\eqref{eq:euler_av_phi}, the
relation~\eqref{eq:choice_pb} is needed in order to obtain~\eqref{eq:energy_av_fin}.

For that purpose, we multiply~\eqref{eq:phi2} by $\frac{\avg{\varphi
    u}}{\avg{\varphi}}$ and we rewrite each of the obtained terms. For
the terms also appearing in the Saint-Venant system i.e. corresponding to
the hydrostatic part of the model, we easily obtain
\begin{equation}
\begin{split}
\left( \frac{\partial}{\partial t}\avg{\varphi u} + \frac{\partial}{\partial
  x} \left( \frac{\avg{\varphi u}^2}{\avg{\varphi}} + g
  \avg{\varphi (\eta-z)}\right) + g \avg{\varphi} \frac{\partial
  z_b}{\partial x}\right) \frac{\avg{\varphi u}}{\avg{\varphi}} = & \\
& \hspace*{-7cm} \frac{\partial}{\partial t} \avg{\varphi E \left(z;\frac{\avg{\varphi u}}{\avg{\varphi}},0\right)}+
  \frac{\partial}{\partial x} \avg{\frac{\avg{\varphi
        u}}{\avg{\varphi}} \varphi E \left(z;\frac{\avg{\varphi
          u}}{\avg{\varphi}},0\right)}.
\end{split}
\label{eq:energ_eq1}
\end{equation}
Multiplying~\eqref{eq:phi3} by $\frac{\avg{\varphi w}}{\avg{\varphi}}$
and using~\eqref{eq:phi1}, we obtain the relation
\begin{equation}
\frac{\partial}{\partial
    t}\frac{\avg{\varphi w}^2}{2\avg{\varphi}}  + \frac{\partial}{\partial
    x}\frac{\avg{\varphi u}\avg{\varphi w}^2}{2\avg{\varphi}^2} = \frac{\avg{\varphi w}}{\avg{\varphi}} \left. p_{nh}\right|_b.
\label{eq:energ_eq2}
\end{equation}
And for the contribution of the non-hydrostatic pressure terms of Eq.~\eqref{eq:phi2} over the
energy balance, it comes
\begin{eqnarray}
\left( \frac{\partial}{\partial x}\avg{\varphi p_{nh}} +
  \left. p_{nh}\right|_b\frac{\partial z_b}{\partial x} \right) \frac{\avg{\varphi u}}{\avg{\varphi}}
& = & \frac{\partial}{\partial
x} \frac{\avg{\varphi p_{nh}}\avg{\varphi
    u}}{\avg{\varphi}} - \avg{\varphi
p_{nh}}\frac{\partial}{\partial x} \frac{\avg{\varphi
  u}}{\avg{\varphi}} \nonumber\\
& & + \left. p_{nh} \right|_b \frac{\avg{\varphi u}}{\avg{\varphi}}\frac{\partial z_b}{\partial x} \nonumber\\
& = & \frac{\partial}{\partial
x} \frac{\avg{\varphi p_{nh}}\avg{\varphi
    u}}{\avg{\varphi}} - \frac{\avg{\varphi
p_{nh}}}{\avg{\varphi}}\frac{\partial \avg{\varphi u}}{\partial x}
\nonumber\\
& & + \frac{\avg{\varphi
p_{nh}}\avg{\varphi u}}{\avg{\varphi}^2}\frac{\partial
\avg{\varphi}}{\partial x} \nonumber\\
& &  + \left. p_{nh} \right|_b \frac{\avg{\varphi
    u}}{\avg{\varphi}}\frac{\partial z_b}{\partial x}.
\label{eq:energ_pnh}
\end{eqnarray}
Since the identity
$$\avg{\varphi z} = \frac{\avg{\varphi}}{2} \left( \avg{\varphi} + 2z_b\right),$$
holds, relation~\eqref{eq:phi4} coupled with~\eqref{eq:phi1} reduces to
\begin{equation}
\avg{\varphi w} = - \frac{\avg{\varphi}}{2}\frac{\partial \avg{\varphi
  u}}{\partial x} + \frac{\avg{\varphi u}}{2}\frac{\partial (\avg{\varphi} + 2z_b)}{\partial x},
\label{eq:wbar_varphi}
\end{equation}
and we can rewrite~\eqref{eq:energ_pnh} under the form
\begin{equation}
\begin{split}
\left( \frac{\partial}{\partial x}\avg{\varphi p_{nh}} +
  \left. p_{nh}\right|_b\frac{\partial z_b}{\partial x} \right)
\frac{\avg{\varphi u}}{\avg{\varphi}} = & \frac{\partial}{\partial
x} \frac{\avg{\varphi p_{nh}}\avg{\varphi
    u}}{\avg{\varphi}} + 2\frac{\avg{\varphi
p_{nh}}}{\avg{\varphi}^2}\avg{\varphi w} \\
&  + \left( \left. p_{nh} \right|_b - 2\frac{\avg{\varphi p_{nh}}}{\avg{\varphi}}\right)\frac{\avg{\varphi
    u}}{\avg{\varphi}}\frac{\partial z_b}{\partial x}.
\end{split}
\label{eq:energ_eq3}
\end{equation}
Adding~\eqref{eq:energ_eq1},\eqref{eq:energ_eq2}
and~\eqref{eq:energ_eq3} gives
\begin{equation}
\begin{split}
\frac{\partial}{\partial t} \avg{\varphi E \left(z;\frac{\avg{\varphi u}}{\avg{\varphi}},\frac{\avg{\varphi w}}{\avg{\varphi}}\right)}+&
  \frac{\partial}{\partial x} \avg{\frac{\avg{\varphi
        u}}{\avg{\varphi}} \left( \varphi E \left(z;\frac{\avg{\varphi
      u}}{\avg{\varphi}},\frac{\avg{\varphi w}}{\avg{\varphi}}\right) +\avg{\varphi p_{nh}} \right)} \\
& = \left( \left. p_{nh} \right|_b - 2\frac{\avg{\varphi
      p_{nh}}}{\avg{\varphi}}\right) \left( \frac{
    \avg{\varphi w}}{\avg{\varphi}} + \frac{\avg{\varphi
      u}}{\avg{\varphi}}\frac{\partial z_b}{\partial x}\right).
\end{split}
\label{eq:energy_av_fin_int}
\end{equation}
Using~\eqref{eq:wbar_varphi} we have
$$ \frac{
    \avg{\varphi w}}{\avg{\varphi}} + \frac{\avg{\varphi
      u}}{\avg{\varphi}}\frac{\partial z_b}{\partial x} = -\frac{1}{2}\frac{\partial
    \avg{\varphi u}}{\partial x} + \frac{\avg{\varphi
      u}}{2\avg{\varphi}}\frac{\partial \avg{\varphi}}{\partial
    x} = -\frac{\avg{\varphi}}{2} \frac{\partial}{\partial x} \left(
  \frac{\avg{\varphi u}}{\avg{\varphi}}\right),$$
and therefore the right hand side of~\eqref{eq:energy_av_fin_int} vanishes iff~\eqref{eq:choice_pb} holds
that concludes the proof.
\end{proof}

\subsection{The proposed non-hydrostatic averaged model and other writings}
\label{subsec:model_nh}

In the following, we no more handle variables corresponding to vertical means of the
solution of the Euler
equations~\eqref{eq:euler}. We adopt the
notation $\overline{f}=f(x,t)$.
By analogy with~\eqref{eq:euler_av11}-\eqref{eq:energy_av_fin}, we consider
as non-hydrostatic averaged model the following system
\begin{subnumcases}{\label{eq:model_proppos}}
\frac{\partial H}{\partial t} + \frac{\partial}{\partial x} \bigl(H\overline{u}\bigr) = 0, \label{eq:euler_11}\\
\frac{\partial}{\partial t}(H\overline{u}) + \frac{\partial}{\partial x}\left(H \overline{u}^2
+ \frac{g}{2}H^2 + H\overline{p}_{nh}\right) = - (gH +
2 \overline{p}_{nh})\frac{\partial
  z_b}{\partial x},\label{eq:euler_22}\\
\frac{\partial}{\partial t}(H\overline{w})  + \frac{\partial}{\partial
  x}(H\overline{w}\overline{u}) = 2 \overline{p}_{nh},\label{eq:euler_33}\\
\frac{\partial}{\partial t}\left(\frac{\eta^2 - z_b^2}{2}\right)  + \frac{\partial}{\partial x}\left(\frac{\eta^2-z_b^2}{2}\overline{u}\right) = H \overline{w}.\label{eq:euler_44}
\end{subnumcases}
The smooth solutions $H$, $\overline{u}$, $\overline{w}$, $\overline{p}_{nh}$
of the system~\eqref{eq:model_proppos} also satisfies the energy balance
\begin{equation}
\frac{\partial \overline{E}}{\partial t} + \frac{\partial}{\partial x}
\left(\overline{u}\bigl(\overline{E}+\frac{g}{2}H^2 + H\overline{p}_{nh}\bigr)\right) = 0,
\label{eq:euler_55}
\end{equation}
where
\begin{equation}
\overline{E}=\frac{H(\overline{u}^2+\overline{w}^2)}{2}+ \frac{gH(\eta+z_b)}{2}.
\label{eq:Ebar}
\end{equation}
Notice that simple manipulations of
Eqs.~\eqref{eq:model_proppos} lead to the relation
\begin{equation}
H \overline{w} = - \frac{H}{2}\frac{\partial (H\overline{u})}{\partial x} +
\frac{H \overline{u}}{2}\frac{\partial (H+2z_b)}{\partial x},
\label{eq:wbar1}
\end{equation}
corresponding to a shallow water expression of the divergence free
condition.

The system~\eqref{eq:model_proppos}-\eqref{eq:euler_55} has been obtained
by one of the authors in~\cite{JSM_M3AS} but in the framework of
asymptotic expansion. In this case, the justification of the closure
relations is less obvious than using the energy-based optimality criterion~\eqref{eq:minSW1}.


Simple manipulations in the equations
of~\eqref{eq:model_proppos} lead to different
formulations of the model which are given in the two following corollaries.
\begin{corollary}
The system \eqref{eq:model_proppos} can be rewritten
under the form
\begin{subnumcases}{\label{eq:model_av_final}}
\frac{\partial H}{\partial t} + \frac{\partial}{\partial x} \bigl(H\overline{u}\bigr) = 0,\label{eq:model_av_final1}\\
\frac{\partial}{\partial t}(H\overline{u}) + \frac{\partial}{\partial
  x} \left( H \overline{u}^2 + \frac{g}{2}H^2 + H\overline{p}_{nh}
\right) = - (gH + 2\overline{p}_{nh})\frac{\partial z_b}{\partial x},\label{eq:model_av_final2}\\
\frac{\partial}{\partial t}\left(H\overline{w}\right)  +
\frac{\partial}{\partial
  x}\left(H\overline{w}\, \overline{u}\right) = 2\overline{p}_{nh},\label{eq:model_av_final3}\\
H \overline{w} = - \frac{H}{2}\frac{\partial (H\overline{u})}{\partial x} +
\frac{H \overline{u}}{2}\frac{\partial (H+2z_b)}{\partial x},\label{eq:model_av_final4}
\end{subnumcases}
and for smooth solutions Eq.~\eqref{eq:euler_55} remains valid.
\label{prop:eul1}
\end{corollary}
\begin{corollary}
The system \eqref{eq:model_proppos} can be rewritten
under the form
\begin{subnumcases}{\label{eq:eulerbis}}
\frac{\partial H}{\partial t} + \frac{\partial}{\partial x} \bigl(H\overline{u}\bigr) = 0,\label{eq:eulerbis_1}\\
\frac{\partial}{\partial t}(H\overline{u}) + \frac{\partial}{\partial
  x} \left( H \overline{u}^2 + \frac{g}{2}H^2 + H\overline{p}_{nh}
\right) = - (gH + 2\overline{p}_{nh})\frac{\partial z_b}{\partial x},\label{eq:eulerbis_2}\\
\frac{\partial}{\partial t}\left(\frac{H^2}{2}\overline{w}\right)  +
\frac{\partial}{\partial
  x}\left(\frac{H^2}{2}\overline{w}\, \overline{u}\right) = H\overline{p}_{nh}
+ H\overline{w}^2  -H\overline{u}\, \overline{w}\frac{\partial
  z_b}{\partial x},\label{eq:eulerbis_3}\\
\frac{\partial}{\partial t}\left(\frac{H^2}{2}\right)  + \frac{\partial}{\partial x}\left(\frac{H^2}{2}\overline{u}\right) = H \overline{w} -H\overline{u}\frac{\partial
  z_b}{\partial x},\label{eq:eulerbis_4}
\end{subnumcases}
and for smooth solutions Eq.~\eqref{eq:euler_55} remains valid.
\label{prop:eulbis}
\end{corollary}

\begin{corollary}
The system \eqref{eq:model_proppos} can be rewritten
under the form
\begin{subnumcases}{\label{eq:eulerter}}
\frac{\partial H}{\partial t} + \frac{\partial}{\partial x} \bigl(H\overline{u}\bigr) = 0,\label{eq:eulerter_1}\\
\frac{\partial}{\partial t}(H\overline{u}) + \frac{\partial}{\partial x}(H \overline{u}^2) +
\frac{\partial}{\partial x}\left(H\overline{p}\right) = - 2\overline{p}\frac{\partial
  z_b}{\partial x},\label{eq:eulerter_2}\\
\frac{\partial}{\partial t}\left(\frac{\eta^2 - z_b^2}{2}\overline{w}\right)  +
\frac{\partial}{\partial
  x}\left(\frac{\eta^2 - z_b^2}{2}\overline{w}\, \overline{u}\right) = (H+2z_b)\overline{p}
+ H\overline{w}^2 - g \frac{\eta^2 - z_b^2}{2},\label{eq:eulerter_3}\\
\frac{\partial}{\partial t}\left(\frac{\eta^2 - z_b^2}{2}\right)  + \frac{\partial}{\partial x}\left(\frac{\eta^2 - z_b^2}{2}\overline{u}\right) = H \overline{w},\label{eq:eulerter_4}
\end{subnumcases}
and for smooth solutions Eq.~\eqref{eq:euler_55} remains valid.
\label{prop:eulter}
\end{corollary}
\begin{proof}[Proofs of corollaries~\ref{prop:eul1},~\ref{prop:eulbis} and~\ref{prop:eulter}]
Equation~\eqref{eq:eulerbis_3} can be obtained
multiplying\break Eq.~\eqref{eq:euler_33} by $\frac{H}{2}$ and
using~\eqref{eq:wbar1} and simple manipulations allow to
obtain~\eqref{eq:eulerbis_4} from~\eqref{eq:euler_44}.
Equation~\eqref{eq:eulerter_3} can be obtained
multiplying Eq.~\eqref{eq:euler_33} by $\frac{H+2z_b}{2}$ and
using~\eqref{eq:wbar1}.
\end{proof}

\red{\begin{remark}
When considering the bottom $z_b$ can vary w.r.t. time $t$, the
system~\eqref{eq:model_proppos} remains unchanged only the energy balance~\eqref{eq:euler_55} is
modified and becomes
\begin{equation}
\frac{\partial \overline{E}}{\partial t} + \frac{\partial}{\partial x}
\left(\overline{u}\bigl(\overline{E}+\frac{g}{2}H^2 +
  H\overline{p}_{nh}\bigr)\right) =
(gH+2\overline{p}_{nh})\frac{\partial z_b}{\partial t},
\label{eq:energ_zb_t}
\end{equation}
with $\overline{E}$ defined by~\eqref{eq:Ebar}. Since
$\left.\overline{p}\right|_b = gH+2\overline{p}_{nh}$, the contributions of
the time variations of $z_b$ in Eq.~\eqref{eq:energ_zb_t} are consistent with those
appearing in~\eqref{eq:energy_eq_euler}.
\label{rem:zb_t}
\end{remark}}

\subsection{About asymptotic expansion}

For shallow water flows, the model derivation is often carried out
using the shallow water assumption. Indeed, introducing
the small parameter
$$\varepsilon = \frac{h}{\lambda},$$
where $h$ and $\lambda$, two characteristic dimensions along the $z$
and $x$ axis respectively, an asymptotic expansion of the Euler or
Navier-Stokes system leads to simplified averaged models such as the
Saint-Venant system. As in~\cite{gerbeau,decoene,marche,JSM_M3AS} and
neglecting the viscous and friction effects, the
shallow water assumption allows to justify the estimate
\begin{equation}
u = \overline{u} + \mathcal{O}(\varepsilon^2),
\label{eq:ubar}
\end{equation}
leading, using the divergence free condition, to
\begin{equation}
w = -(z-z_b)\frac{\partial \overline{u}}{\partial x} +
\overline{u}\frac{\partial z_b}{\partial x} + \mathcal{O}(\varepsilon^2).
\label{eq:westim}
\end{equation}
Inserting~\eqref{eq:ubar} and~\eqref{eq:westim} in the momentum
equation~\eqref{eq:eul_2d3} implies that the non-hydrostatic part of the
pressure is linear in the variable $z$
$$\frac{\partial p_{nh}}{\partial z} = \alpha(x,t)(z-z_b) + \beta(x,t) + \mathcal{O}(\varepsilon^2).$$
Unfortunately, the preceding relation is not compatible with the closure relation
for the pressure~\eqref{eq:choice_pb}. And it is then necessary to add
a scaling coefficient over the non-hydrostatic pressure terms in order
to ensure the existence of an energy balance.

Notice that the energy balance obtained using the rescaled
non-hydrostatic pressure terms differ from~\eqref{eq:energy_av_fin} and~\eqref{eq:euler_55}. The
Green-Naghdi~\cite{green} can be derived using such an asymptotic expansion strategy.

\subsection{Comparison with Green-Naghdi model}
\label{subsec:gn}

One of the most popular models for the description of long, dispersive
water waves is the Green-Naghdi model. Several
derivations of the Green-Naghdi model have been proposed in the
litterature~\cite{green,green1,su,miles}. For the mathematical justification of the model, the reader can refer
to~\cite{lannes,makarenko} and for its numerical approximation to~\cite{lemetayer,chazel1,chazel2,JSM_CF}.

Following~\cite{lemetayer} and with $z_b=cst$, the Green-Naghdi model reads
\begin{subnumcases}{\label{eq:gn}}
\frac{\partial H}{\partial t} + \frac{\partial}{\partial x} \bigl(H\overline{u}\bigr) = 0,\label{eq:gn_1}\\
\frac{\partial (H\overline{u})}{\partial t} +
\frac{\partial}{\partial x} \left( H\overline{u}^2 +
  \frac{g}{2}H^2 + H \overline{p}_{gn}\right) = 0,\label{eq:gn_2}
\end{subnumcases}
with $\overline{p}_{gn} = \frac{1}{3}H \ddot{H}$ and the ``dot'' notation means
the material derivative
\begin{equation}
\dot{H} = \frac{\partial H}{\partial t} + \overline{u} \frac{\partial
  H}{\partial x}.
\label{eq:mat_deriv}
\end{equation}

When $z_b = cst$, the Green-Naghdi model and the non-hydrostatic model~\eqref{eq:model_proppos}
are identical up to a multiplicative constant for the non-hydrostatic pressure. Indeed starting from the expression of
$\overline{p}_{gn}$, the relations~\eqref{eq:gn_1} and
\eqref{eq:mat_deriv} give
\begin{equation*}
\begin{split}
\overline{p}_{gn} = & \frac{1}{3}H \left( \frac{\partial
    \dot{H}}{\partial t} + \overline{u}\frac{\partial
    \dot{H}}{\partial x} \right)\\
 = & \frac{1}{3}H \left( \frac{\partial}{\partial t}\left(
    -H\frac{\partial \overline{u}}{\partial x}\right) + \overline{u}\frac{\partial}{\partial x}\left(
    -H\frac{\partial \overline{u}}{\partial x}\right)\right).
\end{split}
\end{equation*}
If we denote, as in~\eqref{eq:wbar1}
\begin{equation}
\overline{w} = -\frac{H}{2} \frac{\partial \overline{u}}{\partial x},
\label{eq:div_gn}
\end{equation}
it comes
$$\overline{p}_{gn} = \frac{2}{3}H \left( \frac{\partial \overline{w}}{\partial t} +
  \overline{u}\frac{\partial \overline{w}}{\partial x}\right) = \frac{2}{3}\left(
  \frac{\partial}{\partial t} (H \overline{w}) +
  \frac{\partial }{\partial x} (H\overline{u}\overline{w})\right)
.$$
Therefore, the Green-Naghdi can also be written under the form
\begin{subnumcases}{\label{eq:gn3}}
\frac{\partial H}{\partial t} + \frac{\partial}{\partial x} \bigl(H\overline{u}\bigr) = 0,\label{eq:gn_11}\\
\frac{\partial (H\overline{u})}{\partial t} +
\frac{\partial}{\partial x} \left( H\overline{u}^2 +
  \frac{g}{2}H^2 + H \overline{p}_{gn}\right) = 0,\label{eq:gn_22}\\
  \frac{\partial}{\partial t} (H \overline{w}) +
  \frac{\partial }{\partial x} (H\overline{u}\overline{w})= \frac{3}{2}\overline{p}_{gn}\label{eq:gn_33},
\end{subnumcases}
with the constraint~\eqref{eq:div_gn} and completed, for smooth solutions, by the energy balance
\begin{equation}
\frac{\partial \overline{E}_{gn}}{\partial t} + \frac{\partial}{\partial x} \overline{u}
\left( \overline{E}_{gn} + H\overline{p}_{gn}\right) = 0,
\label{eq:energy_gn}
\end{equation}
with
\begin{equation}
\overline{E}_{gn} = \frac{H}{2}\left(\overline{u}^2 +
  \frac{2}{3}\overline{w}^2 \right) + \frac{g}{2}H^2.
\label{eq:energy_exp_gn}
\end{equation}
The energy balance~\eqref{eq:energy_gn} illustrates the
main difference between the Green-Nagdhi model and the proposed
non-hydrostatic model~\eqref{eq:model_proppos}-\eqref{eq:euler_55}. In
the case of a flat bottom, \eqref{eq:Ebar}
and~\eqref{eq:energy_exp_gn} only differ by the coefficient $\frac{2}{3}$ in the vertical part of the kinetic energy.


\red{To summarize, for flat bottom, choosing either $\gamma=2$ or $\gamma=\frac{3}{2}$, the system
\begin{subnumcases}{\label{eq:general}}
\frac{\partial H}{\partial t} + \frac{\partial}{\partial x} \bigl(H\overline{u}\bigr) = 0,\label{eq:general1}\\
\frac{\partial (H\overline{u})}{\partial t} +
\frac{\partial}{\partial x} \left( H\overline{u}^2 +
  \frac{g}{2}H^2 + H \overline{p}\right) = 0,\label{eq:general2}\\
  \frac{\partial}{\partial t} (H \overline{w}) +
  \frac{\partial }{\partial x} (H\overline{u}\overline{w})=
  \gamma\overline{p}\label{eq:general3},\\
\overline{w} = -\frac{H}{2} \frac{\partial \overline{u}}{\partial x},\label{eq:general4}
\end{subnumcases}
corresponds to the depth-averaged system~\eqref{eq:model_proppos} or to the
Green-Naghdi
system~\eqref{eq:div_gn}-\eqref{eq:gn3}, respectively. The system~\eqref{eq:general}
is completed with the energy balance
\begin{equation}
\frac{\partial \overline{E}_{\gamma}}{\partial t} + \frac{\partial}{\partial x} \overline{u}
\left( \overline{E}_{\gamma} + H\overline{p}\right) = 0,
\label{eq:energy_gamma}
\end{equation}
with
\begin{equation}
\overline{E}_{\gamma} = \frac{H}{2}\left(\overline{u}^2 +
  \frac{1}{\gamma}\overline{w}^2 \right) + \frac{g}{2}H^2.
\label{eq:energy_exp_gamma}
\end{equation}
}

Despite its similarities with the Green-Naghdi model, the
non-hydrostatic model\break \eqref{eq:model_proppos}-\eqref{eq:euler_55} has several
advantages
\begin{itemize}
\item its derivation is more simple than the Green-Naghdi model
  (see~\cite{green,green1}),
\item the topography source terms appear quite naturally (that is not
  the case for most of the versions available in the
  literature~\cite{camassa,nadiga}),
\item the model formulation is written under the form of an
  advection-reaction set of PDE and does not contain high order derivatives.
\end{itemize}
\red{A comparison between the solutions of the two non-hydrostatic models is obviously
  a key point but it requires a numerical scheme for their
  discretization that is not in the scope of this paper. We illustrate in
  paragraphs~\ref{subsec:thacker_nh} and~\ref{subsec:solitary} the
  differences between the two non-hydrostatic models in the case of
  analytical solutions.}

\subsection{Hydrostatic case}
\label{subsec:hydro}

The process used for the derivation of the non-hydrostatic model in
paragraph~\ref{subsec:depth-averaging} can also be used for the
derivation of shallow water hydrostatic models.

The hydrostatic assumption in~\eqref{eq:euler} that means that the contribution of the vertical acceleration in the pressure $p$ can be neglected, leads to the classical model
\begin{subnumcases}{\label{eq:eul_hyd}}
\frac{\partial \varphi}{\partial t} + \frac{\partial \varphi
  u}{\partial x} + \frac{\partial \varphi w}{\partial z} = 0,\label{eq:eul_hyd1}\\
\frac{\partial u}{\partial t} + \frac{\partial u^2}{\partial x} + \frac{\partial uw}{\partial z} +  \frac{\partial p}{\partial x} = 0,\label{eq:eul_hyd2}\\
\frac{\partial p}{\partial z} = -g. \label{eq:eul_hyd3}
\end{subnumcases}
This hydrostatic model --~or some variants with horizontal and
vertical viscosity or other specific terms~-- is often used in
geophysical flows studies and it has been widely studied, let us mention some
important contributions \cite{brenier,grenier,masmoudi}.

Starting from Eqs.~\eqref{eq:eul_hyd}, the
shallow water assumption allows to derive the classical Saint-Venant system (see also \cite{saleri,gerbeau,marche})
\begin{subnumcases}{\label{eq:sv}}
\frac{\partial H}{\partial t} + \frac{\partial}{\partial x} \bigl(H\overline{u}\bigr) = 0,\label{eq:sv_1}\\
\frac{\partial (H\overline{u})}{\partial t} + \frac{\partial (H\overline{u}^2)}{\partial x} + \frac{g}{2}\frac{\partial H^2}{\partial x} = -gH \frac{\partial z_b}{\partial x}.\label{eq:sv_2}
\end{subnumcases}
The smooth solutions of \eqref{eq:sv}
satisfy the energy equality
\begin{eqnarray}
\frac{\partial E_h}{\partial t} + \frac{\partial}{\partial x}
\left(\overline{u}\bigl(E_h+g\frac{H^2}{2}\bigr)\right) = 0,\label{eq:sv_3}
\end{eqnarray}
with the energy
\begin{equation}
E_h = \frac{H\overline{u}^2}{2}+\frac{gH(\eta+z_b)}{2}.
\label{eq:sv_4}
\end{equation}
Notice that (\ref{eq:sv_3}),(\ref{eq:sv_4}) corresponds to
(\ref{eq:energy_exp}),(\ref{eq:energy_eq_euler}) where the hydrostatic and
shallow water assumptions are made.

\subsection{A depth-averaged Navier-Stokes system}
\label{sec:NS_av}

In Section~\ref{sec:av_euler}, we have started from the Euler system to
obtain its depth-averaged version. In this section, we use the same
process as in
paragraphs~\ref{sec:av_euler}
to obtain a depth-averaged Navier-Stokes system. And we have the
following proposition
\begin{proposition}
A depth-averaged version of the free surface Navier-Stokes system leads to the model
\begin{subnumcases}{\label{eq:NS}}
\frac{\partial H}{\partial t} + \frac{\partial}{\partial x} \bigl(H\overline{u}\bigr) = 0, \label{eq:NS_11}\\
\frac{\partial}{\partial t}(H\overline{u}) + \frac{\partial}{\partial x}\left(H \overline{u}^2
+ \frac{g}{2}H^2 + H\overline{p}_{nh}\right) = \nonumber\\
\hspace*{4cm} - (gH +
2 \overline{p}_{nh})\frac{\partial
  z_b}{\partial x} + \frac{\partial}{\partial x}\left(
  2\mu H\frac{\partial \overline{u}}{\partial x}\right) - \kappa \overline{u},\label{eq:NS_22}\\
\frac{\partial}{\partial t}(H\overline{w})  + \frac{\partial}{\partial
  x}(H\overline{w}\overline{u}) = 2 \overline{p}_{nh} +
\frac{\partial}{\partial x}\left(\mu
  H\frac{\partial \overline{w}}{\partial x}\right),\label{eq:NS_33}\\
\frac{\partial}{\partial t}\left(\frac{\eta^2 - z_b^2}{2}\right)  + \frac{\partial}{\partial x}\left(\frac{\eta^2-z_b^2}{2}\overline{u}\right) = H \overline{w}.\label{eq:NS_44}
\end{subnumcases}
Moreover the smooth solutions of
\eqref{eq:NS} satisfy the energy
balance
\begin{equation}
\begin{split}
\frac{\partial \overline{E}}{\partial t} + \frac{\partial}{\partial x}
\left(\overline{u}\left(\overline{E}+\frac{g}{2}H^2 + H\overline{p}_{nh} - 2\mu H
  \frac{\partial \overline{u}}{\partial x}\right) - \mu H
  \overline{w}\frac{\partial \overline{w}}{\partial x}\right) & \\
& \hspace*{-4cm} =  -\mu H \left(
 2 \left(\frac{\partial \overline{u}}{\partial x} \right)^2 +
  \left(\frac{\partial \overline{w}}{\partial x} \right)^2\right) -
\kappa \overline{u}^2,
\end{split}
\label{eq:NS_7}
\end{equation}
with $\overline{E}$ defined by~\eqref{eq:Ebar}.
\label{prop:NS1}
\end{proposition}

\begin{proof}[Proof of proposition~\ref{prop:NS1}]
Compared to the derivation of the model~\eqref{eq:model_proppos}-\eqref{eq:euler_55}, only the treatment of the viscous
terms has to be precised and we have
$$
\int \left(\frac{\partial \Sigma_{xx}}{\partial x} + \frac{\partial
  \Sigma_{xz}}{\partial z}\right) \varphi dz =
\frac{\partial}{\partial x}\int 2\mu \frac{\partial u}{\partial x}
\varphi dz - \kappa \overline{u},$$
where the boundary conditions~\eqref{eq:BC2},\eqref{eq:BC_z_b} have
been used. And replacing $u$ by $\overline{u}$ in the r.h.s. of the
preceding relation gives the expression of the viscous term
in~\eqref{eq:NS_22}. Likewise, using~\eqref{eq:BC2},\eqref{eq:BC_z_b},
we have
$$
\int \left(\frac{\partial \Sigma_{zx}}{\partial x} + \frac{\partial
  \Sigma_{zz}}{\partial z}\right) \varphi dz =
\frac{\partial}{\partial x}\int \mu \frac{\partial w}{\partial x}
\varphi dz,$$
and replacing $w$ by $\overline{w}$ gives the expression of the viscous term
in~\eqref{eq:NS_33}. Multiplying~\eqref{eq:NS_22} by $\overline{u}$
and~\eqref{eq:NS_33} by $\overline{w}$ and after simple manipulations,
we obtain the relation~\eqref{eq:NS_7} that completes the proof.
\end{proof}

\section{Some properties of the non-hydrostatic model}
\label{sec:properties}

\subsection{Expression for $\overline{p}_{nh}$}

Equation~\eqref{eq:euler_44}~-- that is equivalent
to~\eqref{eq:wbar1}~-- is not a dynamical equation but a constraint
ensuring a shallow water version of the divergence free
condition. And hence it plays a specific role in the non-hydrostatic
model. We try to reformulate Eq.~\eqref{eq:wbar1} in order to obtain an
equation satisfied by the pressure $\overline{p}_{nh}$. The process
used is similar to Chorin solenoidal decomposition of the velocity
field \cite{chorin} for Navier-Stokes equations.

The derivative w.r.t. time $t$ of the shallow water form of the divergence free condition~\eqref{eq:wbar1}
gives
$$\frac{\partial (H\overline{w})}{\partial t} + \frac{H}{2}\frac{\partial^2
  (H\overline{u})}{\partial x \partial t} - \frac{1}{2}\frac{\partial
  (H+2z_b)}{\partial x}\frac{\partial
  (H\overline{u})}{\partial t} = - \frac{H\overline{u}}{2}\frac{\partial^2
  (H\overline{u})}{\partial x^2} + \frac{1}{2} \left(\frac{\partial
  (H\overline{u})}{\partial x}\right)^2,$$
where relation~\eqref{eq:euler_11} has been used. Now substituting the expressions~\eqref{eq:euler_22},\eqref{eq:euler_33} for
$$\frac{\partial (H\overline{u})}{\partial t} ,\quad\mbox{and}\quad \frac{\partial (H\overline{w})}{\partial t},$$
in the previous relation gives
\begin{equation}
2 \overline{p}_{nh} + \left( \frac{1}{2}\frac{\partial
  \left(H+2z_b \right)}{\partial x} - \frac{H}{2}\frac{\partial}{\partial x}\right)\left(\frac{\partial
  (H\overline{p}_{nh})}{\partial x}+2 \overline{p}_{nh} \frac{\partial z_b}{\partial x}\right) = B,
\label{eq:d2p}
\end{equation}
with
\begin{eqnarray*}
B & = & \frac{1}{2} \left(\frac{\partial
  (H\overline{u})}{\partial x}\right)^2 - \frac{H\overline{u}}{2}\frac{\partial^2
  (H\overline{u})}{\partial x^2} + \frac{H}{2}\left(
  \frac{\partial^2}{\partial x^2} \left(H \overline{u}^2 + \frac{g}{2}H^2\right) + g
\frac{\partial}{\partial x}\left(H\frac{\partial z_b}{\partial
  x}\right) \right) \\
& & + \frac{\partial
  (H\overline{w}\overline{u})}{\partial x} - \frac{1}{2}\frac{\partial
  \left(H+2z_b\right)}{\partial x}\left( \frac{\partial}{\partial
    x}\left(H\overline{u}^2 + \frac{g}{2}H^2\right)
  +gH\frac{\partial z_b}{\partial x}\right).
\end{eqnarray*}
From~\eqref{eq:wbar1}, we get
\begin{eqnarray*}
\frac{\partial (H\overline{w}\overline{u})}{\partial x} & = & -\frac{1}{2}
\frac{\partial}{\partial x} \left( H\overline{u}\frac{\partial
    (H\overline{u})}{\partial x}\right) + \frac{\partial}{\partial x}
\left( \frac{H\overline{u}^2}{2} \frac{\partial (H+2z_b)}{\partial x}\right),
\end{eqnarray*}
leading to
\begin{eqnarray*}
B & = & - H\overline{u}\frac{\partial^2
  (H\overline{u})}{\partial x^2} + \frac{H}{2}\left(
  \frac{\partial^2}{\partial x^2} \left(H \overline{u}^2 + \frac{g}{2}H^2\right) + g
\frac{\partial}{\partial x}\left(H\frac{\partial z_b}{\partial
  x}\right) \right) \nonumber\\
& & + \frac{H\overline{u}^2}{2} \frac{\partial^2 (H+2z_b)}{\partial x^2} - \frac{1}{2}\frac{\partial
  \left(H+2z_b\right)}{\partial x}\left( \frac{\partial}{\partial
    x}\left(\frac{g}{2}H^2\right)
  +gH\frac{\partial z_b}{\partial x}\right)\nonumber\\
& = & H\left( -\overline{u}\frac{\partial^2
  (H\overline{u})}{\partial x^2} + \frac{1}{2}\frac{\partial^2
  (H\overline{u}^2)}{\partial x^2} + \frac{\overline{u}^2}{2} \frac{\partial^2 (H+2z_b)}{\partial x^2}\right) \nonumber\\
& & + \frac{gH}{2} \left(H \frac{\partial^2 (H+z_b)}{\partial
      x^2} - 2\frac{\partial z_b}{\partial x}\frac{\partial (H+z_b)}{\partial x} \right).
\end{eqnarray*}
Introducing the new variable
$$\overline{q}_{nh} = \sqrt{H}\overline{p}_{nh},$$
relation~\eqref{eq:d2p} becomes
\begin{equation}
- 4H^2\frac{\partial^2 \overline{q}_{nh}}{\partial x^2}  + \Lambda \overline{q}_{nh} =
8\sqrt{H} B,
\label{eq:d2p_new}
\end{equation}
that is an non-homogeneous differential equation with
$$\Lambda = 16 \left( 1 + \left(\frac{\partial z_b}{\partial
      x}\right)^2\right) -8H\frac{\partial^2 z_b}{\partial x^2} + 16
\frac{\partial H}{\partial x}\frac{\partial z_b}{\partial x}-
2H\frac{\partial^2 H}{\partial x^2} + 3 \left(\frac{\partial H}{\partial x}\right)^2.$$
And the sign of $\Lambda$ in Eq.~\eqref{eq:d2p_new} gives interesting
informations about the influence of the non-hydrostatic terms. Indeed,
for smooth/small variations of $z_b$ and $H$, we have $\Lambda > 0$
whereas large variations of $z_b$ and $H$ can lead to the situation
where $\Lambda < 0$.

When $\Lambda>0$, Eq.~\eqref{eq:d2p_new} corresponds to a diffusion type
equation and when
$\Lambda<0$, Eq.~\eqref{eq:d2p_new} corresponds to an
Helmholtz type equation. This remark is very important since situations
where $\Lambda<0$ may correspond to areas where the non-hydrostatic effects
can be significant

\subsection{Requirements for the pressure $\overline{p}$}

The positivity of the pressure $p$ for the incompressible Euler
equations (see~paragraph~\ref{subsec:assump}) is an acute problem. On
the one hand, the Euler system allows the pressure $p$ to be
non-positive, on the other hand $p<0$ means that the fluid is no more in
contact with the bottom and the system~\eqref{eq:euler}-\eqref{eq:ps},\eqref{eq:free_surf},\eqref{eq:bottom} has to be reformulated,
especially its boundary conditions.

This problem vanishes when considering the Saint-Venant
system. Indeed in this situation, the pressure term corresponds to
$$\frac{g}{2}H^2,$$
that is always non-negative.

When $H \rightarrow 0$ the Euler equations, the proposed
non-hydrostatic model but also the Saint-Venant system are no more physically
relevant. We would like in this situation, as for the Saint-Venant
system, that the model~\eqref{eq:model_proppos}-\eqref{eq:euler_55} well behaves both at the continuous and
discrete level.

\section{Analytical solutions}
\label{sec:anal_sol}

The analysis of the proposed non-hydrostatic model being very complex, the
knowledge of analytical solutions allows to examine the behavior of
the model in particular situations. Moreover, analytical solutions are
an important tool for the validation of numerical schemes.

In the following, we propose different analytical solutions for the averaged
non-hydrostatic model~\eqref{eq:model_proppos}-\eqref{eq:euler_55}.

\subsection{Time dependent analytical solutions}
\label{subsec:thacker_nh}

In this paragraph we consider the Euler system~\eqref{eq:euler} with
the boundary conditions~\eqref{eq:free_surf},\eqref{eq:bottom}
and~\eqref{eq:ps}. This system can also be written under the form
\begin{subnumcases}{\label{eq:euler_2d_mod}}
\frac{\partial H}{\partial t} + \frac{\partial}{\partial
  x}\int_{z_b}^\eta u\ dz = 0,\label{eq:euler_2d_mod11}\\
w = -\frac{\partial}{\partial x}\int_{z_b}^z u\ dz,\label{eq:euler_2d_mod22}\\
\frac{\partial {u}}{\partial {t}} + u\frac{\partial {u}}{\partial
  {x}} + w\frac{\partial
  {u}}{\partial {z}} + \frac{\partial {p}}{\partial {x}} = 0,\label{eq:euler_2d_mod33}\\
\frac{\partial {w}}{\partial {t}} + u\frac{\partial {w}}{\partial
  {x}} + w\frac{\partial
  {w}}{\partial {z}} + \frac{\partial {p}}{\partial {z}} = -g + s,\label{eq:euler_2d_mod44}
\end{subnumcases}
coupled with the boundary condition~\eqref{eq:ps} where $s$ is an external
forcing term.

And we have the following proposition.

\begin{proposition}
Let us consider the variables $u,w,H,z_b,p$ defined by
\begin{subnumcases}{\label{eq:sol}}
H(x, t) = \max \left( H_0 -\frac{b_2}{2} \left( x-\int^t_{\tilde{t}^0}
    f(t_1)dt_1\right)^2, 0 \right), \label{eq:sol1}\\
u(x,z,t) = f(t)\1_{H>0}, \label{eq:sol2}\\
w(x,z,t) = b_2 x f(t) \1_{H>0}, \label{eq:sol3}\\
z_b(x) = b_1+\frac{b_2}{2} x^2, \label{eq:sol4}\\
p(x,z,t) =  (g + b_2f^2)(H+z_b - z) \1_{H>0}, \label{eq:sol5}\\
s(x,z,t) = b_2 x \frac{d f}{ d t},  \label{eq:sol6}
\end{subnumcases}
where $H_0>0,b_1,b_2$ are constants and the function $f$ satisfies the ODE
\begin{equation}
\frac{d f}{dt} +b_2 (g+b_2 f^2) \int^t_{\tilde{t}^0}
f(t_1)dt_1= 0,\quad f(t_0) = f^0, \;\; \tilde{t}^0\in \R.
\label{eq:f}
\end{equation}
Then $u,w,H,z_b,p$ as defined previously
satisfy the 2d incompressible Euler equations with free
surface~\eqref{eq:euler_2d_mod} with the boundary
condition~\eqref{eq:ps} where $p^a=0$.
\label{prop:thacker_euler2d}
\end{proposition}
\begin{proof}
The proof relies on simple manipulations. Replacing~\eqref{eq:sol}
in~\eqref{eq:euler_2d_mod} shows the solution is analytic
when~\eqref{eq:f} is satisfied.
\end{proof}
\begin{remark}
Analytical solutions without the source term $s$
in~\eqref{eq:euler_2d_mod44} would have been a stronger
result. Nevertheless, since we only consider a source term for one of
the four
equations~\eqref{eq:euler_2d_mod}, it
remains an interesting result for numerical validations.
\end{remark}
These analytical solutions generalize the solutions obtained by
Thacker~\cite{thacker} for the shallow water equations. The analysis of the ODE~\eqref{eq:f} is not in the scope of this
paper.
Notice that the change of variables
$$h(t) = \int_{\tilde{t}^0}^t f(t_1)dt_1,$$
allows to
rewrite~\eqref{eq:f} under the form
\begin{equation}
\begin{split}
& \frac{d^2 h}{dt^2} +b_2 \left(g+b_2 \left(\frac{d h}{d t}\right)^2
\right) h = 0,\\
& h(t_0) = \int_{\tilde{t}^0}^{t^0} f(t_1)dt_1, \;\;
\dot{y}(t_0) = f(t_0) = f^0.
\end{split}
\label{eq:f2}
\end{equation}
It is worth noticing that when $H>0$ the free surface is a straight line varying with
time. Indeed, from the definitions of prop.~\ref{prop:thacker_euler2d}
and when $H>0$,
we get that for any $t$
$$H+z_b = b_1 -\frac{b_2}{2} \left( -2x \int^t_{\tilde{t}^0}
  f(t_1)dt_1
+ \left(\int^t_{\tilde{t}^0} f(t_1)dt_1\right)^2 \right),$$
that is a linear function of the $x$ variable.

The analytical solution depicted in prop.~\eqref{prop:thacker_euler2d}
is interesting for two reasons. First, it allows to confront a
numerical scheme with behaviors difficult to capture typically drying
and flooding. The second reason is explained in the following
proposition.

\begin{proposition}
The variables $H$, $\overline{u}$, $\overline{w}$, $z_b$ defined as in
Eqs.~\eqref{eq:sol1}-\eqref{eq:sol4} and $\overline{p}$
defined by
$$\overline{p} = \frac{g}{2}H + \overline{p}_{nh} = \frac{1}{H}\int_{z_b}^\eta p(x,z,t)
dz,$$
with $p$ given in~\eqref{eq:sol5} are analytical solutions of the
depth-averaged Euler system~\eqref{eq:model_proppos} completed with the
source term $s$.
\label{prop:thacker_svnh}
\end{proposition}

The propositions~\ref{prop:thacker_euler2d}
and~\ref{prop:thacker_svnh} produce a very important
consequence. Taking into account the source term $s$, we
have exhibited an analytical solution for the 2d Euler system~\eqref{eq:euler}-\eqref{eq:energy_eq_euler} with
free surface which is also an analytical solution for the
non-hydrostatic model~\eqref{eq:model_proppos}-\eqref{eq:euler_55} we
propose. This a strong argument proving our model is a good
approximation of the Euler system for shallow water flows. And this is
reinforced by the following proposition.

\red{
\begin{proposition}
When $f$ satisfies~\eqref{eq:f}, the solution~\eqref{eq:sol} is not an analytical solution of the
Green-Naghdi
model~\eqref{eq:gn3}-\eqref{eq:energy_gn}. If $f$ satisfies the ODE
\begin{equation}
\frac{d f}{dt} +b_2 \left( g+\frac{4 b_2}{3} f^2 \right) \int^t_{\tilde{t}^0}
f(t_1)dt_1= 0,\quad f(t_0) = f^0, \;\; \tilde{t}^0\in \R,
\label{eq:f_mod}
\end{equation}
then~\eqref{eq:sol} is an analytical solution of the Green-Naghdi
model~\eqref{eq:gn3}-\eqref{eq:energy_gn}. But the energy balance~\eqref{eq:energy_gn} is not consistent with the energy equation~\eqref{eq:energy_eq_euler} of the Euler system.
\label{prop:sol_anal_gn}
\end{proposition}
\begin{proof}[Proof of prop.~\ref{prop:sol_anal_gn}]
The proof relies on simple calculations. Since from~\eqref{eq:sol} we have
\begin{eqnarray*}
\int_{z_b}^\eta E\ dz & = & \int_{z_b}^\eta \left( \frac{u^2+w^2}{2} +
  gz \right) dz = -\frac{b_0 f^2}{4} \left( x -\int^t_{\tilde{t}^0} f(t_1)
  dt_1 \right)^{2}
 \left(  1 + b_0^{2} x^2
 \right) \\
& & + \frac{gb_0^2}{8} \left(
 \left( x^2 -\left( x -\int^t_{\tilde{t}^0} f(t_1) dt_1
 \right)^{2} \right)^{2}- x^4 \right)
\end{eqnarray*}\begin{eqnarray*}& = & \overline{E}\\
& \neq & \overline{E}_{gn},
\end{eqnarray*}
this proves the result.
\end{proof}
To illustrate the difference between the solutions having the
form~\eqref{eq:sol} for the two non-hydrostatic
models~\eqref{eq:model_proppos} and~\eqref{eq:gn3}, we plot on
Fig.~\ref{fig:f_f_mod} the solutions of~\eqref{eq:f} and~\eqref{eq:f_mod}.
\begin{figure}[htbp]
\begin{center}
\includegraphics[width=8.1cm]{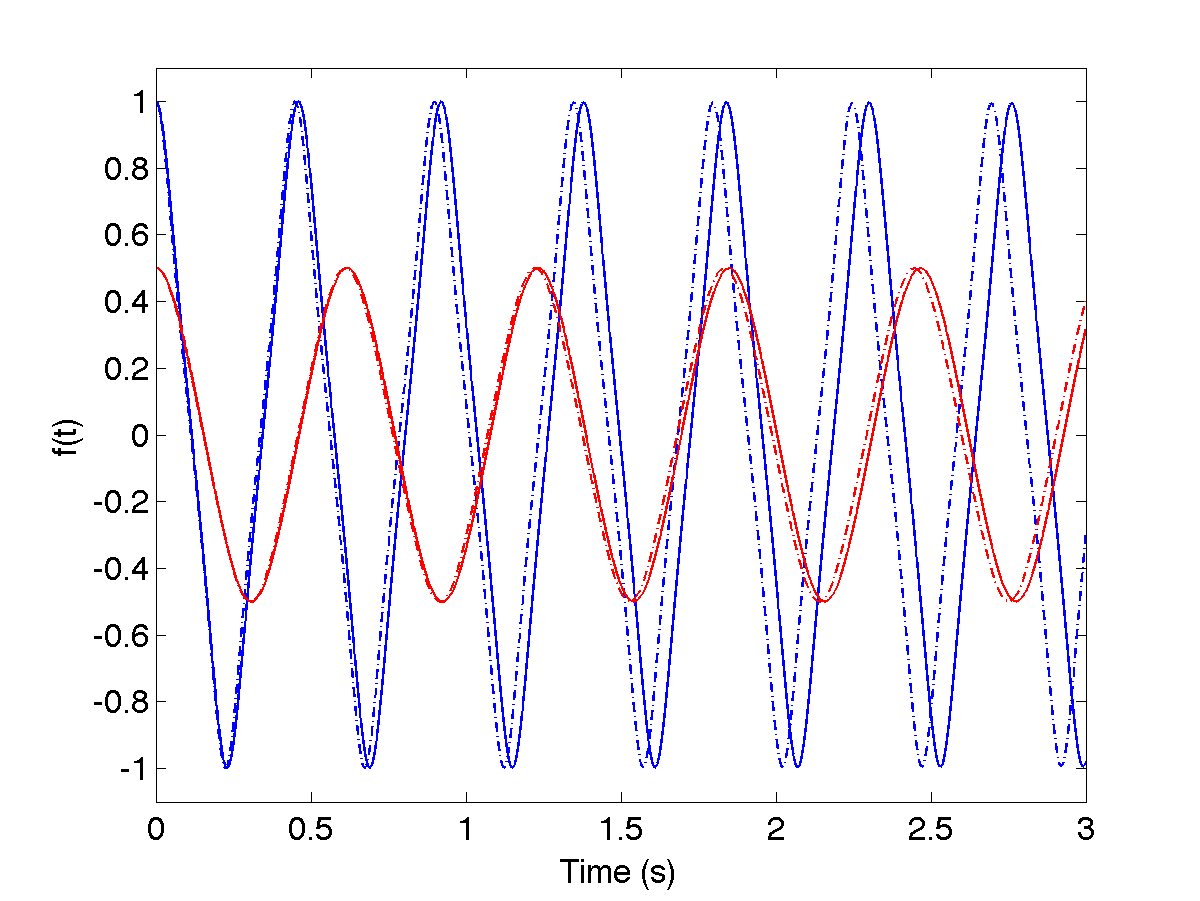}
\end{center}
\caption{Comparison of the solutions of ~\eqref{eq:f} (solid lines)
  and~\eqref{eq:f_mod} (dashed lines).}
\label{fig:f_f_mod}
\end{figure}
The solutions have been obtained using an implicit first order Euler
scheme solving~\eqref{eq:f} and~\eqref{eq:f_mod}. Over
Fig.~\ref{fig:f_f_mod}, the solid lines (resp. the dashed lines) correspond to solutions
of~\eqref{eq:f} (resp.~\eqref{eq:f_mod}). The two curves with
amplitude $1$ have been obtained with $b_0=15$, $\tilde{t}_0=t_0=0$ s,
$f(t_0)=1$ and the two curves with
amplitude $\frac{1}{2}$ have been obtained with $b_0=10$, $\tilde{t}_0=t_0=0$ s,
$f(t_0)=\frac{1}{2}$. We observe that whenever the solutions
of~\eqref{eq:f} and~\eqref{eq:f_mod} remain periodic, the period
differs especially for large values of $b_0$.
}

\subsection{Solitary wave solutions}
\label{subsec:solitary}

Using a process similar to what is done in~\cite{lemetayer,chazel2},
in the case where $z_b=cst$, we can exhibit solitary waves for the
system~\eqref{eq:model_proppos} under the form
\begin{subnumcases}{\label{eq:soliton}}
H =  H_0 + a \left( \sech\left(\frac{x-c_0t}{l}\right) \right)^2,\label{eq:soliton1}\\
\overline{u}  =  c_0 \left( 1 - \frac{d}{H} \right),\label{eq:soliton2}\\
\overline{w} = -\frac{ac_0 d}{l H}\sech \left(\frac{x-c_0t}{l}\right)
\sech' \left(\frac{x-c_0t}{l}\right),\label{eq:soliton3}\\
\overline{p}_{nh} = \frac{ac_0^2d^2}{2l^2H^2} \left( (2H_0-H)
  \left(\sech' \left(\frac{x-c_0t}{l}\right)\right)^2\right. \nonumber\\
\qquad\left. + H \sech \left(\frac{x-c_0t}{l}\right) \sech'' \left(\frac{x-c_0t}{l}\right)\right),\label{eq:soliton4}
\end{subnumcases}
where $f'$ denotes the derivative of function $f$ and
$$c_0 = \frac{l}{d}\sqrt{\frac{g H^3_0}{l^2-H^2_0}},\quad\mbox{and}\quad
a=\frac{H^3_0}{l^2-H^2_0},$$
and $(d,l,H_0) \in \R^3$ with $l>H_0>0$.

\red{The system~\eqref{eq:soliton} also gives analytical solutions for
  the Green-Naghdi system. Indeed, replacing $a$ and $c_0$ by
  $a_\gamma$ and $c_{0,\gamma}$ defined by
\begin{eqnarray*}
a_\gamma & = & \frac{H^3_0}{\frac{\gamma}{2} l^2-H^2_0},\\
c_{0,\gamma} & = & \sqrt{\frac{\gamma}{2}}\frac{l}{d}\sqrt{\frac{g H^3_0}{\frac{\gamma}{2} l^2-H^2_0}},
\end{eqnarray*}
with $(d,l,H_0) \in \R^3$ and $l>H_0>0$, the
system~\eqref{eq:soliton} gives an analytical solution for the
general system~\eqref{eq:general}.
Therefore, we are able to
compare the analytical solutions of the two non-hydrostatic
system. On Fig.~\ref{fig:soliton_1}, we have plotted the water depth at three
different instants $t_0=0$ s, $t_1=4$ s and $t_2=12$ s corresponding
to the propagation of the two analytical solitary waves with $H_0=1$
m and $d=2$ m. Fig.~\ref{fig:soliton_1}-{\it (a)}, the analytical
solutions of the depth-averaged model and the Green-Naghdi model are
depicted by the solid and dashed lines, respectively. The solutions
correspond to the choice $l=2$ m and the corresponding values of $a_2$
and $c_{0,2}$ for the depth-averaged model and the corresponding
values of $a_{3/2}$
and $c_{0,3/2}$ for the Green-Naghdi model.
We see on Fig.~\ref{fig:soliton_1}-{\it (a)} that starting from the
same physical parameters $H_0,d,l$, the two non-hydrostatic models
propagate two solitons but with different amplitudes and propagation velocities.
On the contrary, we can choose the physical parameters, typically $l$,
so that the two solitons have the same amplitude and propagation
velocities. Indeed, choosing for the depth-averaged system $l=2$ m and the corresponding values of $a_2$
and $c_{0,2}$ and for the Green-Naghdi model $l=\frac{4}{\sqrt{3}}$ m
and the corresponding values of $a_{3/2}$ and $c_{0,3/2}$ we obtained
on Fig.~\ref{fig:soliton_1}-{\it (b)} two solitons with the same
amplitude and propagation velocities but a slightly different shape.
\begin{figure}[htbp]
\begin{center}
\begin{tabular}{c}
\includegraphics[width=8.1cm]{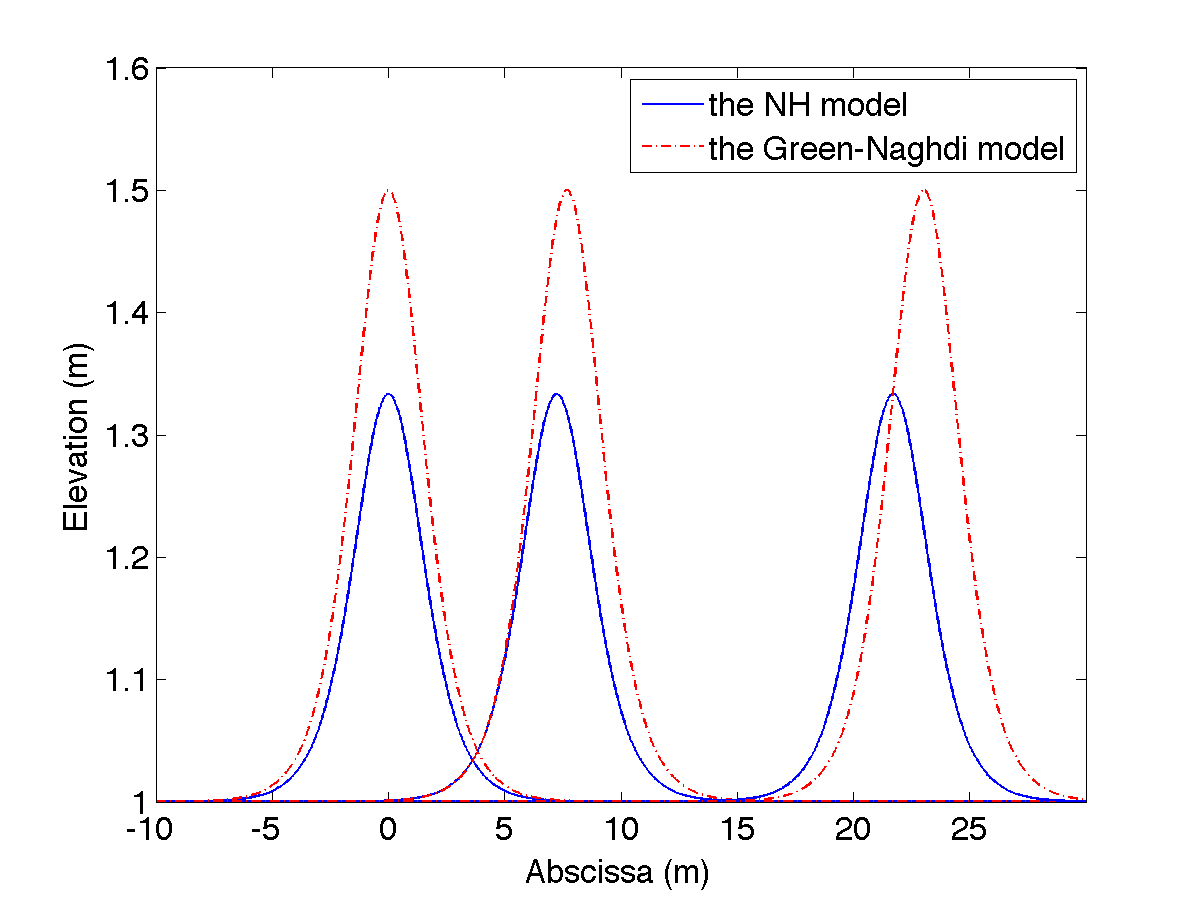}\\
{\it (a)}\\
\includegraphics[width=8.1cm]{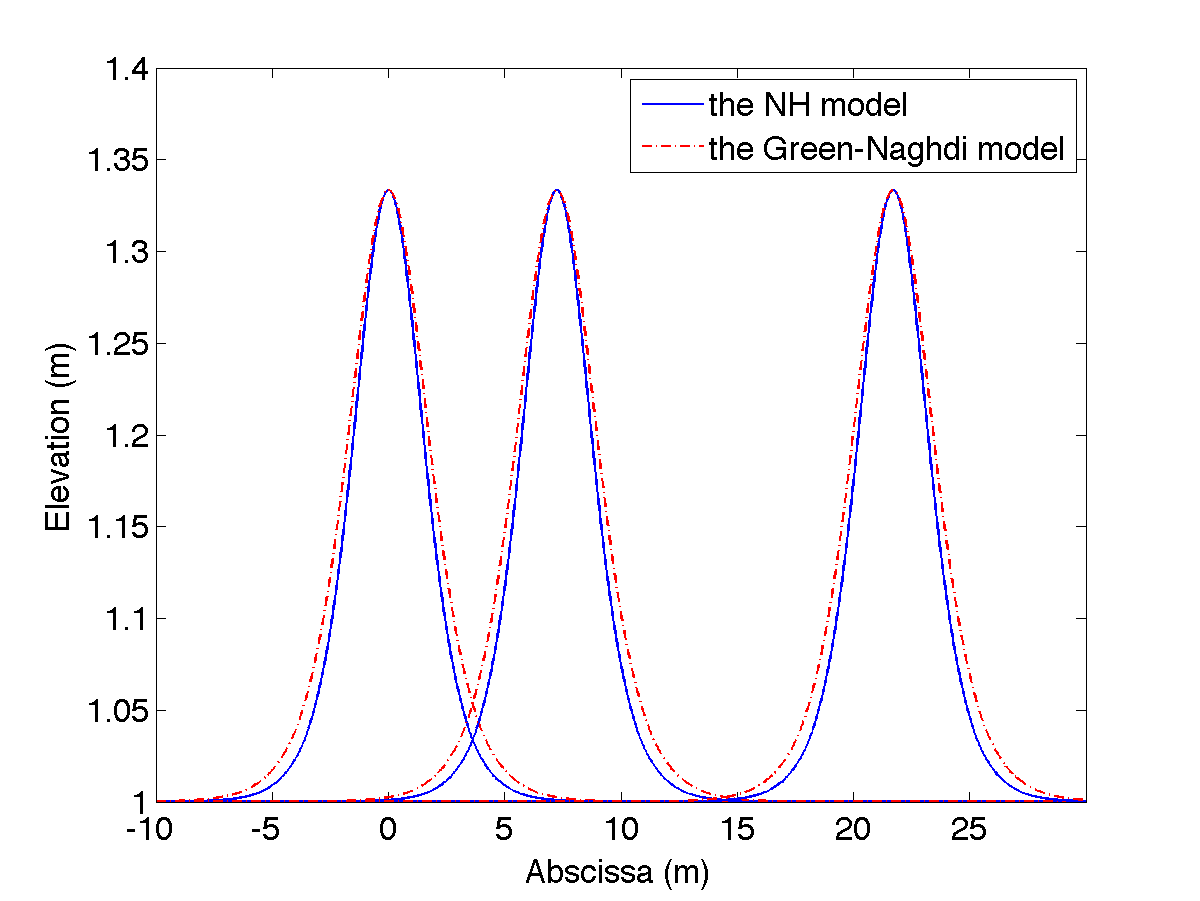}\\
 {\it (b)}
\end{tabular}
\end{center}
\caption{Comparaison of analytical solutions of the depth-averaged
  model (NH) (solid lines) and of the Green-Naghdi (dashed lines) in
  the case of a solitary wave: {\it (a)} same values of $l$ {\it (b)}
  same amplitude and propagation velocity.}
\label{fig:soliton_1}
\end{figure}
}

\subsection{Stationary solutions}
\label{subsec:station_sol}

\subsubsection{Regularity of stationary solutions}

Simple manipulations show that stationary analytical solutions of
\eqref{eq:model_proppos} have to satisfy
\begin{subnumcases}{\label{eq:anal}}
H \overline{u} = Q_0 = Cst,\label{eq:anal_1}\\
\frac{\partial}{\partial x}\left(\frac{Q^2_0}{H} +\frac{g}{2}H^2 + H\overline{p}_{nh}\right) = - \left(gH+2\overline{p}_{nh}\right)\frac{\partial z_b}{\partial x},\label{eq:anal_2}\\
H\overline{w} = \frac{Q_0}{2}\frac{\partial}{\partial x} \left(
  H+2z_b\right),\label{eq:anal_3}\\
\overline{p}_{nh} = \frac{Q_0}{2}\frac{\partial
  \overline{w}}{\partial x}, \label{eq:anal_4}
\end{subnumcases}
or equivalently

\ \vspace*{-10pt}
\begin{subnumcases}{\nonumber}
\frac{\partial H}{\partial x} = \frac{2}{Q_0} H\overline{w} - 2
\frac{\partial z_b}{\partial x}\nonumber\\
\frac{\partial \overline{w}}{\partial x} = \frac{2}{Q_0}
\overline{p}_{nh}, \nonumber\\
\frac{\partial \overline{p}_{nh}}{\partial x} =
  \left( \frac{Q^2_0}{H^2} - gH - \overline{p}_{nh} \right) \left(\frac{2}{Q_0} \overline{w} - \frac{2}{H}
\frac{\partial z_b}{\partial x}
  \right) - \left(g+\frac{2\overline{p}_{nh}}{H}\right)\frac{\partial z_b}{\partial x},\nonumber
\end{subnumcases}
and $\overline{u} = \frac{Q_0}{H}$. Hence, as long as $H>0$, we have $(H,\overline{w},\overline{p}_{nh}) \in (C^k)^3$ if $z_b \in C^k$. This means
that when $z_b$ is at least continuous, the stationary solutions of
the non-hydrostatic model are necessarily continuous and do not
admit shocks.

\subsubsection{Stationary quasi-analytical solutions}

From the previous writing, we deduce the following proposition.
\begin{proposition}
Choosing $Q_0$, a boundary condition $H_0$ for $H$ and a given function $f=f(x)$
corresponding to the desired vertical velocity i.e. $\overline{w}=f$, then the variables
$\overline{p}_{nh} ,H,z_b,\overline{u},$ defined by
\begin{subnumcases}{\label{eq:anal_stat}}
\overline{p}_{nh} = \frac{Q_0}{2}\frac{\partial f}{\partial x},\label{eq:anal_stat_1}\\
\left(\frac{g}{2}H - \frac{Q_0^2}{H^2}\right)\frac{\partial H}{\partial x}  =
  -\frac{H}{Q_0}\left( gH + Q_0 \frac{\partial f}{\partial x}\right) f
  - \frac{Q_0}{2} H \frac{\partial^2 f}{\partial x^2},\label{eq:anal_stat_2}\\
\frac{\partial z_b}{\partial x}   =  -\frac{1}{2}\frac{\partial H}{\partial x}  + \frac{Hf}{Q_0},\label{eq:anal_stat_3}\\
\overline{u} = \frac{Q_0}{H} \label{eq:anal_stat_4},
\end{subnumcases}
are stationary quasi-analytical of the system
\eqref{eq:model_proppos}.

The word ``quasi-analytical''
refers to the fact that the previous set of equations only contains two simple ODEs that have to be solved numerically.
\label{prop:anal_sol}
\end{proposition}

\begin{proof}[Proof of proposition~\ref{prop:anal_sol}]
The proof is very simple, it only consists in a reformulation of the
system~\eqref{eq:anal_1}-\eqref{eq:anal_3} with the assumption $\overline{w}=f$, $f$ given.
\end{proof}

\begin{remark}
Since the quantity
$$\frac{g}{2}H - \frac{Q_0^2}{H^2},$$
appears in the ODE to solve~\eqref{eq:anal_stat_4}, it is possible to obtain solutions
for $H$
with discontinuities. But necessarily, due to the second equation to
solve, discontinuities also appears over $z_b$. Thus, this is not
contradictory with the results in paragraph~\ref{subsec:station_sol}.
\label{rem:discont}
\end{remark}
\red{As in paragraphs~\ref{subsec:thacker_nh} and~\ref{subsec:solitary}, we compare the stationary
  solutions for the depth-averaged model~\eqref{eq:model_proppos} and
  the Green-Naghdi system~\eqref{eq:gn3}. Following prop.~\ref{prop:anal_sol}, we can also
  exhibit stationary quasi-analytical solutions for the Green-Naghdi
  system~\eqref{eq:gn3}. For any given enough smooth function, $f$ and
  the solutions of the system
\begin{subnumcases}{\label{eq:anal_stat_gn}}
\overline{p}_{nh} = \frac{Q_0}{2}\frac{\partial f}{\partial x},\label{eq:anal_stat_1_gn}\\
\left(\frac{g}{2}H - \frac{Q_0^2}{H^2}\right)\frac{\partial H}{\partial x}  =
  -\frac{H}{Q_0}\left( gH + \frac{4Q_0}{3} \frac{\partial f}{\partial x}\right) f
  - \frac{2}{3}Q_0 H \frac{\partial^2 f}{\partial x^2},\label{eq:anal_stat_2_gn}\\
\frac{\partial z_b}{\partial x}   =  -\frac{1}{2}\frac{\partial H}{\partial x}  + \frac{Hf}{Q_0},\label{eq:anal_stat_3_gn}\\
\overline{u} = \frac{Q_0}{H} \label{eq:anal_stat_4_gn},
\end{subnumcases}
are analytical solutions of the Green-Naghdi
system~\eqref{eq:gn3}. Numerical comparisons between the solutions of
systems~\eqref{eq:anal_stat} and~\eqref{eq:anal_stat_gn} are given in the following paragraph.
}

\subsubsection{Numerical illustrations}

To illustrate the analytical solutions described by
prop.~\ref{prop:anal_sol}, we give below two typical examples. The
analytical solutions are obtained choosing
\begin{equation}
f(x) = 2c (x-a) e^{-b (x-a)^2},
\label{eq:f_anal_stat}
\end{equation}
and correspond to a channel of length $L=10$ $m$ where we impose the
inflow $Q_0>0$ at the entrance (left boundary) and the water depth $H_0$ at the
exit (right boundary). For Fig.~\ref{fig:sol_anal_1}, the following
parameters values $Q_0=1.8$ $m^2.s^{-1}$, $H_0=1$ $m$, $a=5$ $m$, $b=3.4$ $m^{-2}$ and $c=1.5$
$s^{-1}$ are considered. On Fig.~\ref{fig:sol_anal_1}-{\it (a)}, we
compare the free surface $\eta=H+z_b$ obtained with the
quasi-analytical solution~\eqref{eq:anal_stat_2},\eqref{eq:anal_stat_4} of the non-hydrostatic model to the one obtained with the Saint-Venant system (with
the same topography $z_b$ and the same boundary conditions). Likewise
on Fig.~\ref{fig:sol_anal_1}-{\it (b)}, we
compare the velocity field $\overline{u}$ obtained with the
depth-averaged Euler model to the one obtained with the Saint-Venant system (with
the same topography $z_b$ and the same boundary conditions). The
velocity field $\overline{w}$ corresponding to the depth-averaged
system is also plotted on~ Fig.~\ref{fig:sol_anal_1}-{\it (b)}. Over
Fig.~\ref{fig:sol_anal_1}-{\it (c)}, we compare the total pressure
$gH/2+\overline{p}_{nh}$ to its hydrostatic part $gH/2$.

Figure~\ref{fig:sol_anal_2} is similar to Figure~\ref{fig:sol_anal_1}
but has been obtained with the parameters values $Q_0=1.35$ $m^2.s^{-1}$,
$a=5$ $m$, $b=4.6$ $m^{-2}$ and $c=1.0$
$m^{-1}$. Figures~\ref{fig:sol_anal_1} and~\ref{fig:sol_anal_2}
emphasize the influence of the non-hydrostatic effects.

\red{On Fig.~\ref{fig:sol_anal_3}, we compare the quasi-analytical
  solutions of the systems~\eqref{eq:anal_stat}
  and~\eqref{eq:anal_stat_gn}. The definition of the function $f$ is
  still given by~\eqref{eq:f_anal_stat}. For the inflow $Q_0$, we have
  chosen $Q_0=1.3$ $m^2.s^{-1}$, the
  boundary condition $H_0$ and the parameters $a$, $b$ and
  $c$ have the same values as for Fig.~\ref{fig:sol_anal_2}. Notice that
  the resolution of~\eqref{eq:anal_stat_gn} with exactly the same
  parameters values as those used for Fig.~\ref{fig:sol_anal_2}
  i.e. with $Q_0=1.35$ $m^2.s^{-1}$ instead of $Q_0=1.3$ $m^2.s^{-1}$
  leads to a discontinuous solution (see remark~\ref{rem:discont}).
In the Green-Naghdi model, the amplitude of the waves is higher than in the depth-averaged model.
}
\begin{figure}[htbp]
\begin{center}
\begin{tabular}{c}
\includegraphics[width=7.5cm]{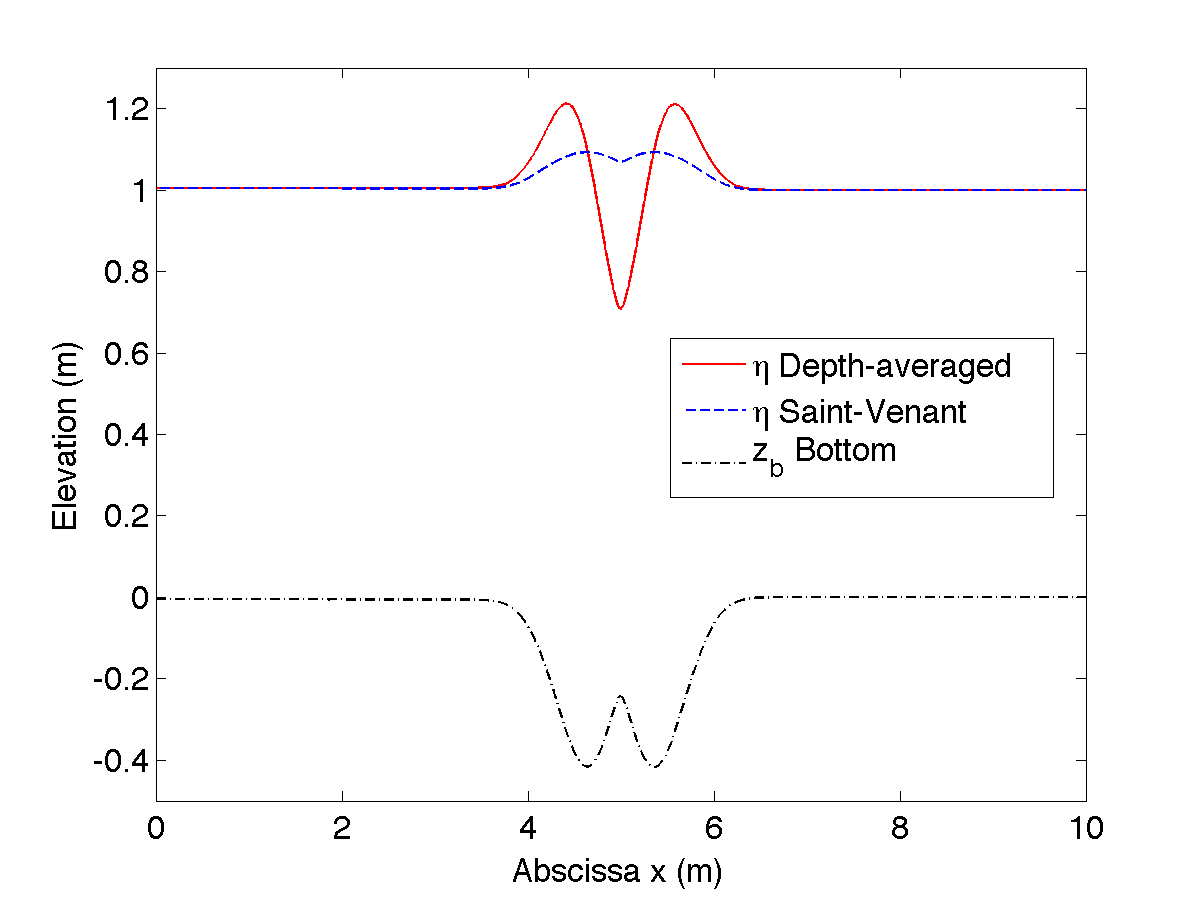}\\
{\it (a)}\\
\includegraphics[width=7.5cm]{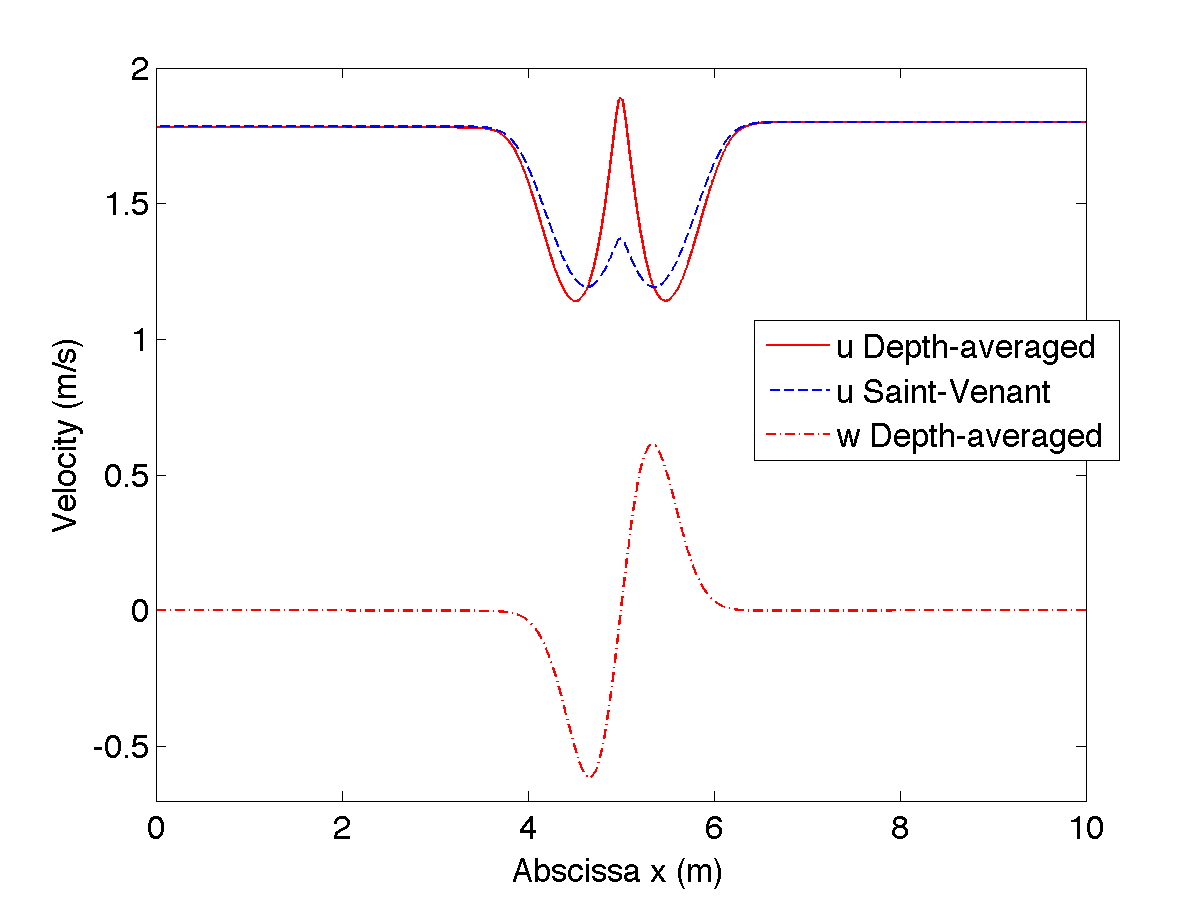}\\
{\it (b)}\\
\includegraphics[width=7.5cm]{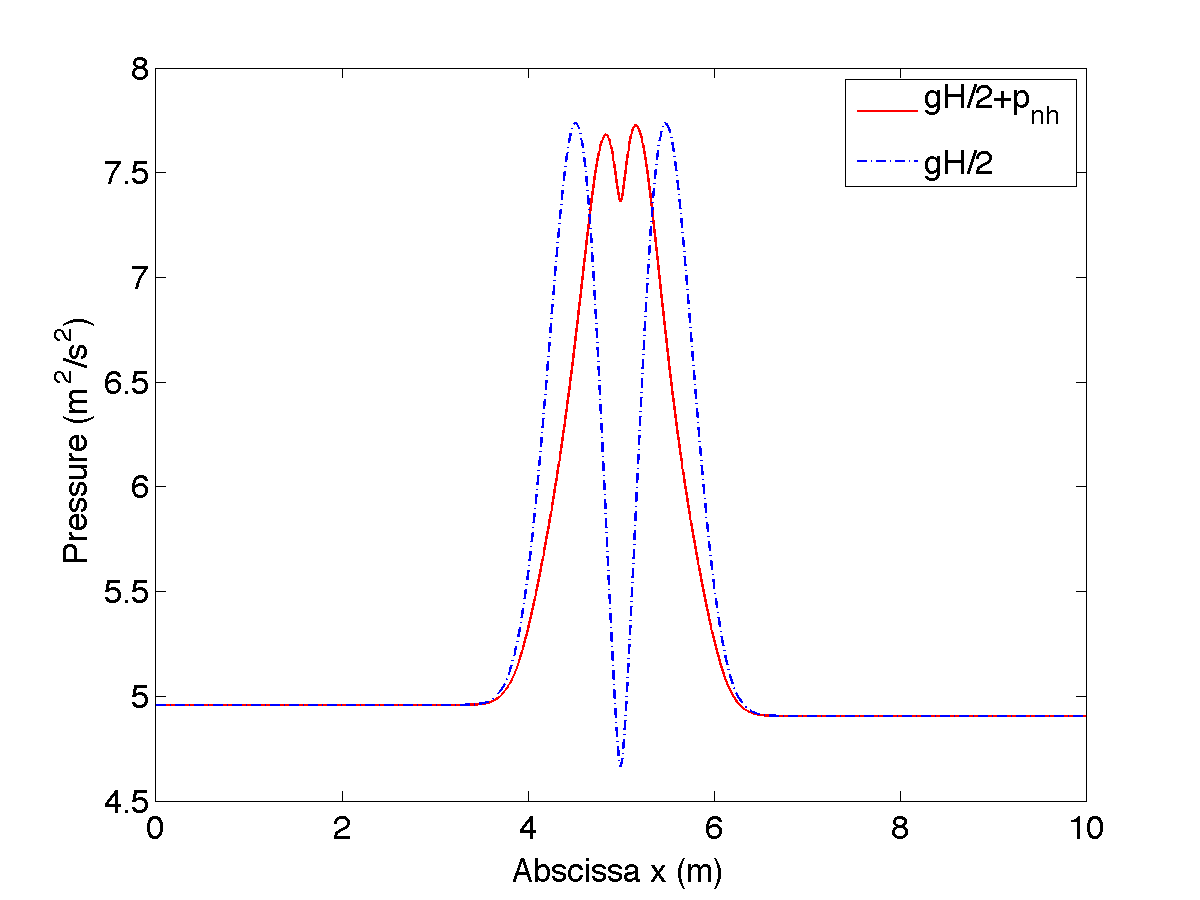}\\
{\it (c)}
\end{tabular}
\end{center}
\caption{Analytical solutions - Comparison of Saint-Venant and depth-averaged Euler
    solutions: {\it (a)} free surface $H+z_b$ and
  bottom profile $z_b$,  {\it (b)} velocities $\overline{u}$ and
  $\overline{w}$ and  {\it (c)} total pressure
  $gH/2+\overline{p}_{nh}$ and hydrostatic part of the pressure $gH/2$.}
\label{fig:sol_anal_1}
\end{figure}

\begin{figure}[htbp]
\begin{center}
\begin{tabular}{c}
\includegraphics[width=7.5cm]{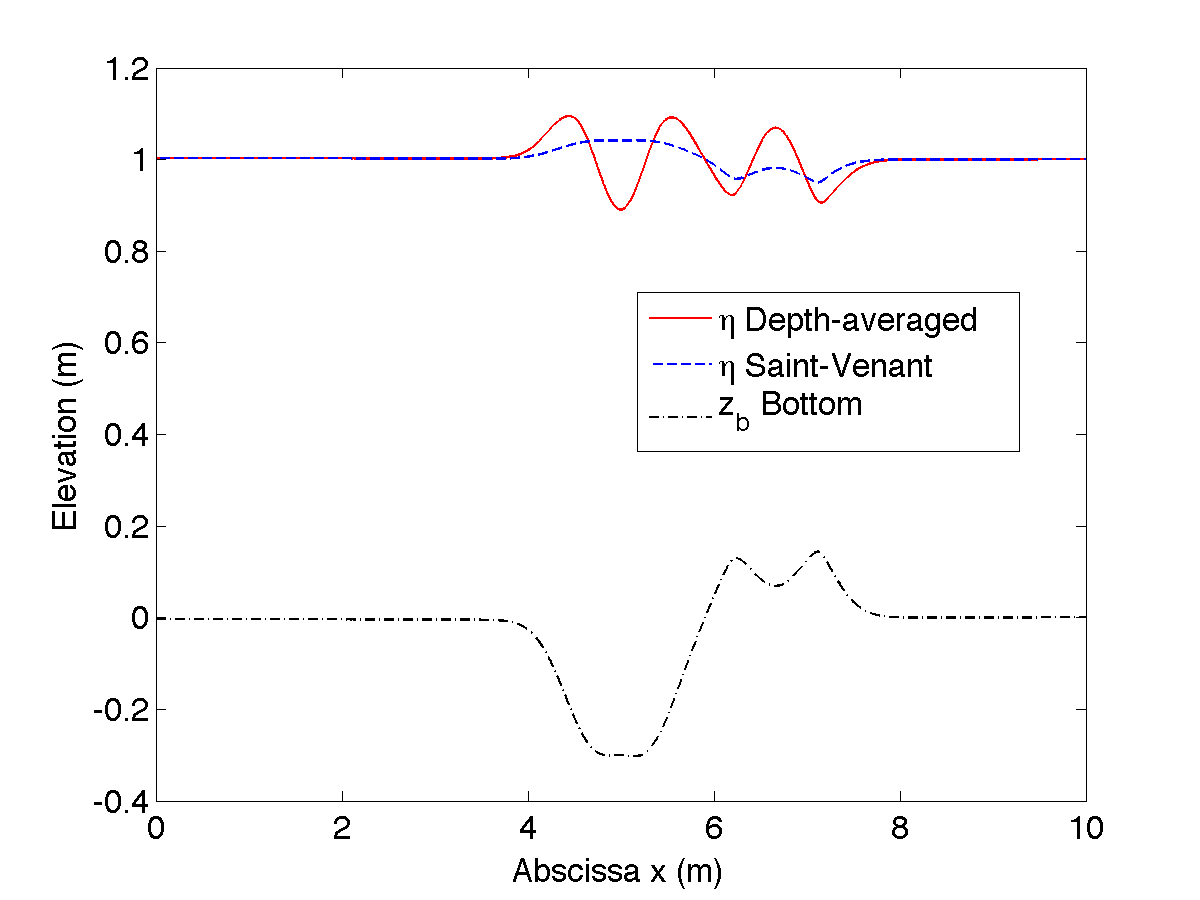}\\
{\it (a)}\\
\includegraphics[width=7.5cm]{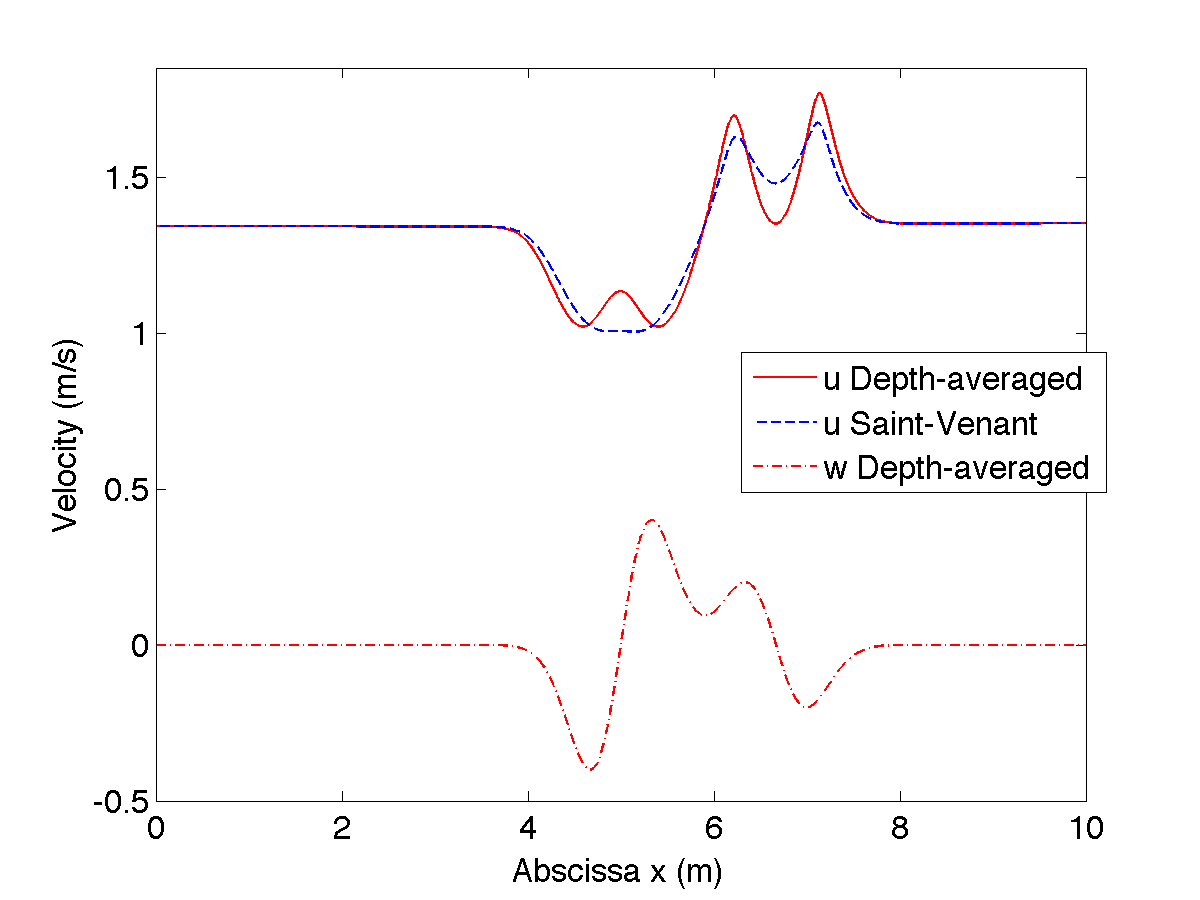}\\
{\it (b)}\\
\includegraphics[width=7.5cm]{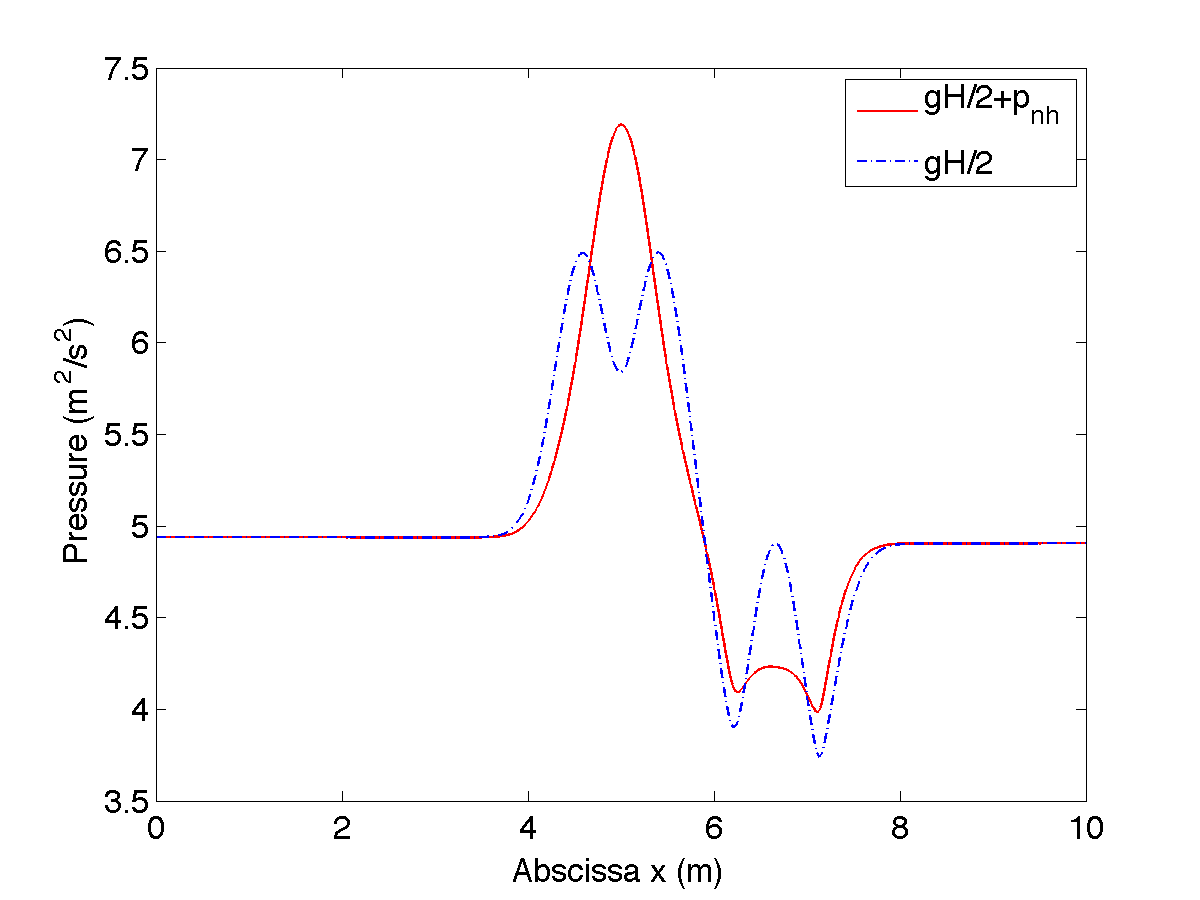}\\
{\it (c)}\\
\end{tabular}
\end{center}
\caption{Analytical solutions- Comparison of Saint-Venant and depth-averaged Euler
    solutions: {\it (a)} free surface $H+z_b$ and
  bottom profile $z_b$,  {\it (b)} velocities $\overline{u}$ and
  $\overline{w}$ and  {\it (c)} total pressure
  $gH/2+\overline{p}_{nh}$ and hydrostatic part of the pressure $gH/2$.}
\label{fig:sol_anal_2}
\end{figure}

\begin{figure}[htbp]
\begin{center}
\begin{tabular}{c}
\includegraphics[width=7.5cm]{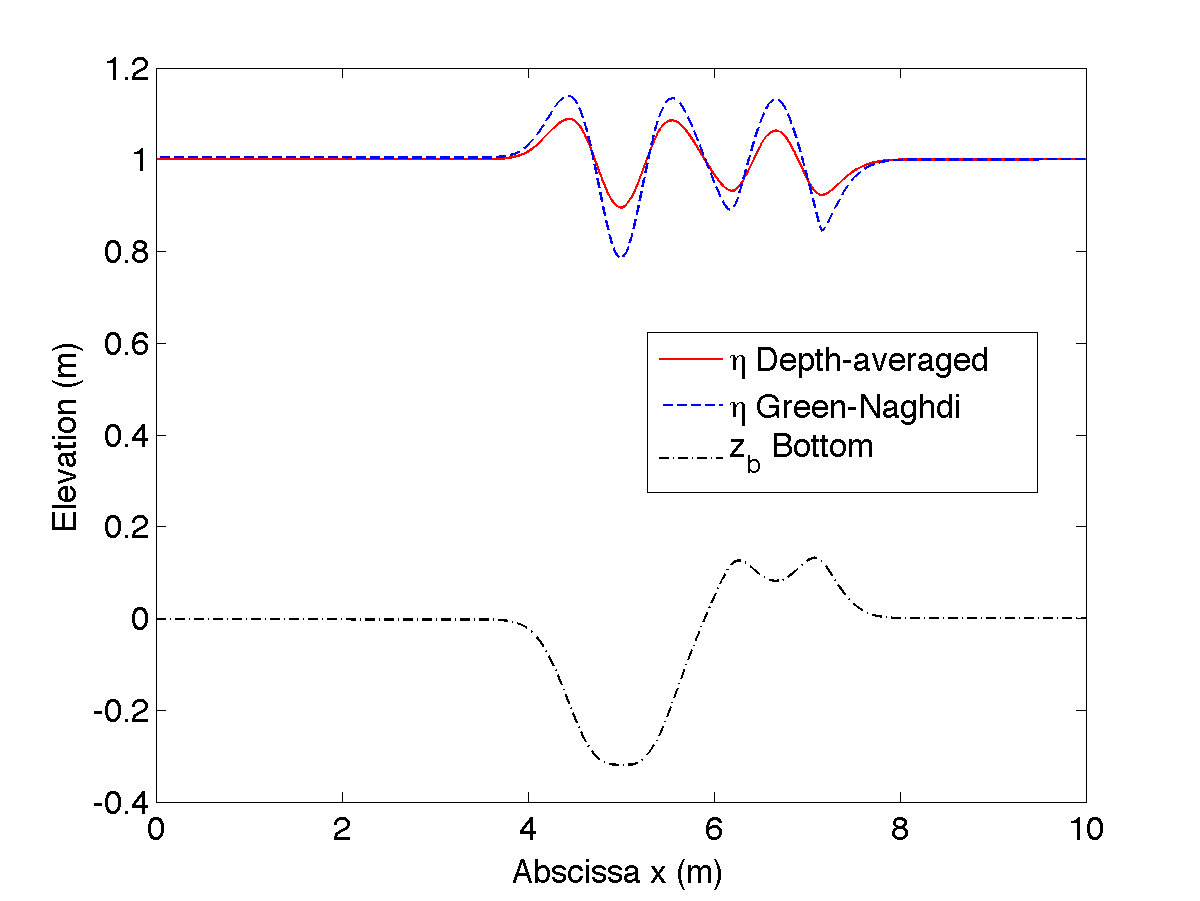}\\
{\it (a)} \\
\includegraphics[width=7.5cm]{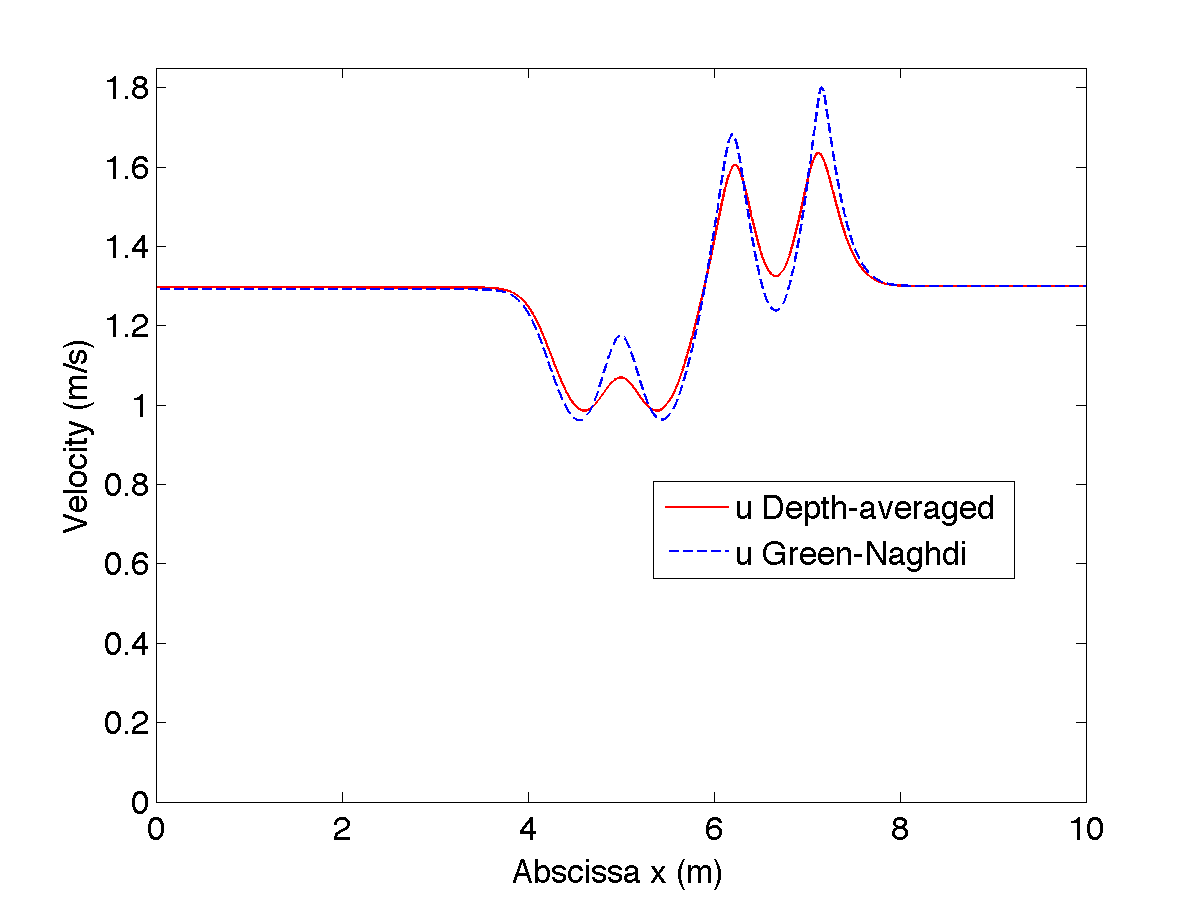}\\
{\it (b)}\\
\includegraphics[width=7.5cm]{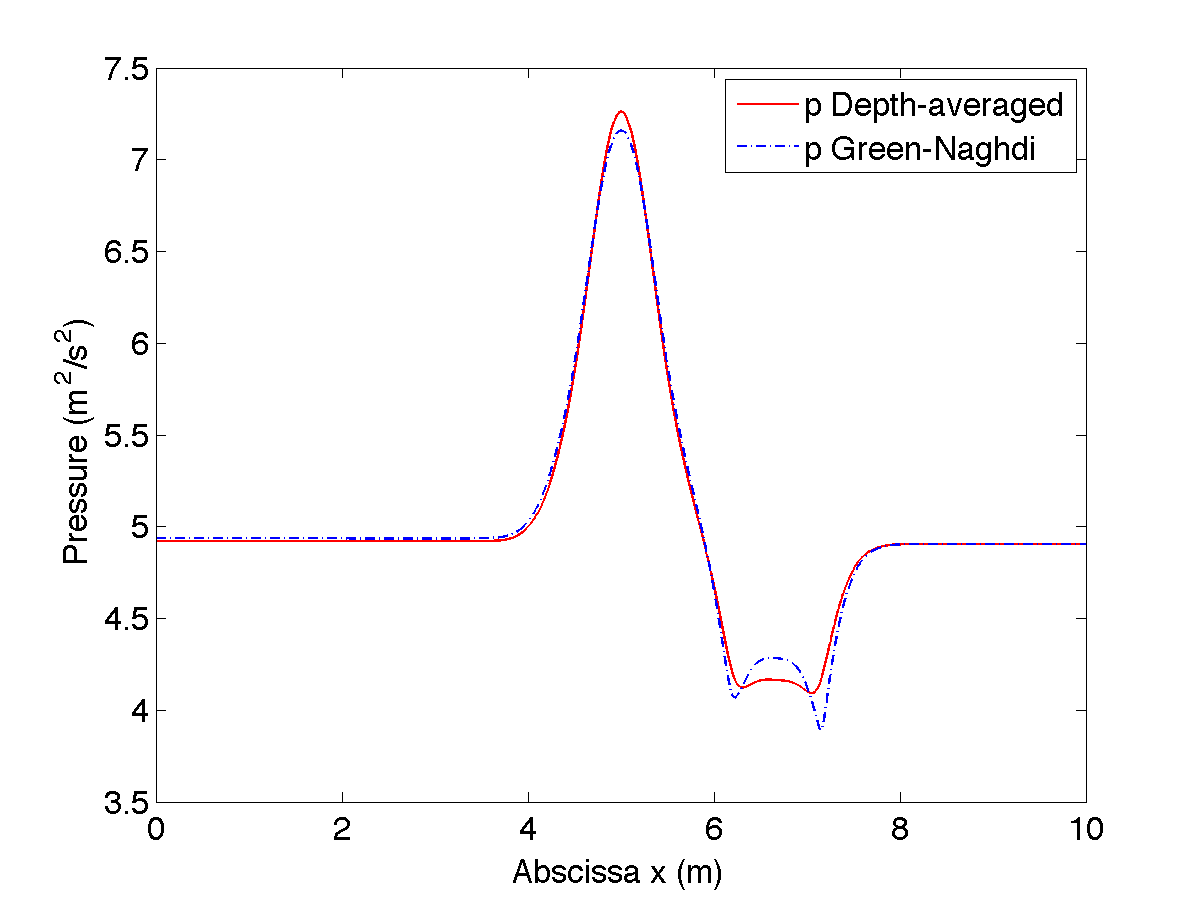}\\
{\it (c)}
\end{tabular}
\end{center}
\caption{Analytical solutions- Comparison of Green-Naghdi and depth-averaged Euler
    solutions: {\it (a)} free surface $H+z_b$ and
  bottom profile $z_b$,  {\it (b)} velocities $\overline{u}$ and  {\it (c)} total pressures
  $gH/2+\overline{p}_{nh}$.}
\label{fig:sol_anal_3}
\end{figure}

\section{Conclusion}

In this paper we have proposed a shallow water type model integrating
the non-hydrostatic effects. The derivation process is based on a
minimization principle and suitable closure relations.

The proposed depth-averaged Euler system has interesting properties
\begin{itemize}
\item the model formulation only involves first order partial
  derivatives,
\item the derivation process naturally provides with an expression for
  the topography source terms,
\item the proposed model is similar to the well-known Green-Naghdi
  model
but gives a natural expression of the topography source term,
\item starting from the Navier-Stokes system instead of the Euler
  system, a depth-averaged version of the Navier-Stokes system is
  obtained integrating the viscous/friction effects.
\end{itemize}
Since the pressure terms are not necessarily non negative, the
behavior of the averaged model when the water depth tends to zero has to be
clarified. The derivation of an
efficient and robust numerical scheme able to treat theses situations is
under study.

\section*{Acknowledgments}

This work was primarily undertaken during the fourth author's
secondment at Inria. Cerema is acknowledged for partial support of the
second author.
The authors also thank Nora A\"issiouene, Emmanuel Audusse, Nicole Goutal
and Benoit Perthame for helpful discussions.

\providecommand{\href}[2]{#2}
\providecommand{\arxiv}[1]{\href{http://arxiv.org/abs/#1}{arXiv:#1}}
\providecommand{\url}[1]{\texttt{#1}}
\providecommand{\urlprefix}{URL }

\medskip
Received April 2014; revised September  2014.
\medskip
\end{document}

%% file: Figures/notations.pdf_t
\begin{picture}(0,0)%
\includegraphics{notations.pdf}%
\end{picture}%
\setlength{\unitlength}{4144sp}%
\begingroup\makeatletter\ifx\SetFigFont\undefined%
\gdef\SetFigFont#1#2#3#4#5{%
  \reset@font\fontsize{#1}{#2pt}%
  \fontfamily{#3}\fontseries{#4}\fontshape{#5}%
  \selectfont}%
\fi\endgroup%
\begin{picture}(9813,5565)(1021,-5923)
\put(7921,-3256){\makebox(0,0)[lb]{\smash{{\SetFigFont{20}{24.0}{\rmdefault}{\mddefault}{\updefault}{\color[rgb]{0,0,0}$u(x,z,t)\approx \overline{u}(x,t)$}%
}}}}
\put(10531,-1726){\makebox(0,0)[lb]{\smash{{\SetFigFont{20}{24.0}{\rmdefault}{\mddefault}{\updefault}{\color[rgb]{0,0,0}$x$}%
}}}}
\put(1036,-601){\makebox(0,0)[lb]{\smash{{\SetFigFont{20}{24.0}{\rmdefault}{\mddefault}{\updefault}{\color[rgb]{0,0,0}$z$}%
}}}}
\put(5716,-1276){\makebox(0,0)[lb]{\smash{{\SetFigFont{20}{24.0}{\rmdefault}{\mddefault}{\updefault}{\color[rgb]{0,0,0}Free surface}%
}}}}
\put(2836,-4741){\makebox(0,0)[lb]{\smash{{\SetFigFont{20}{24.0}{\rmdefault}{\mddefault}{\updefault}{\color[rgb]{0,0,0}$z_b(x,t)$}%
}}}}
\put(5356,-3391){\makebox(0,0)[lb]{\smash{{\SetFigFont{20}{24.0}{\rmdefault}{\mddefault}{\updefault}{\color[rgb]{0,0,0}$H(x,t)$}%
}}}}
\put(5986,-5371){\makebox(0,0)[lb]{\smash{{\SetFigFont{20}{24.0}{\rmdefault}{\mddefault}{\updefault}{\color[rgb]{0,0,0}Bottom}%
}}}}
\put(1036,-1996){\makebox(0,0)[lb]{\smash{{\SetFigFont{20}{24.0}{\rmdefault}{\mddefault}{\updefault}{\color[rgb]{0,0,0}$0$}%
}}}}
\put(3336,-1666){\makebox(0,0)[lb]{\smash{{\SetFigFont{20}{24.0}{\rmdefault}{\mddefault}{\updefault}{\color[rgb]{0,0,0}$H(x,t)+z_b(x)$}%
}}}}
\end{picture}%